\documentclass[11pt]{amsart}
\usepackage{amsmath,amssymb}

\numberwithin{equation}{section}
\newcommand{\V}{\Vert}

\newcommand{\RR} {\mathbb R}
\newcommand{\CC} {\mathbb C}

\newcommand{\pa} {\partial}
\newcommand{\Cal} {\mathcal}

\newcommand{\beq} {\begin{equation}}
\newcommand{\eeq} {\end{equation}}

\newtheorem{theorem}{Theorem}[section]
\newtheorem{remark}[theorem]{ Remark}

\newtheorem{corollary}[theorem]{Corollary}

\newtheorem{proposition}[theorem]{Proposition}

\newtheorem{lemma}[theorem]{Lemma}

\newtheorem{definition}[theorem]{Definition}

\begin{document}
\title[Travelling waves for NLS and NLKG on Riemannian manifolds ]{ Nonlinear travelling waves on non-Euclidean spaces}

\author{Mayukh Mukherjee}
\thanks{The author was partially supported by NSF grant  DMS-1161620.}

\address{Max Planck Institute for Mathematics\\ Vivatsgasse 7\\ 53111 Bonn,
\\ Germany}

\email{mukherjee@mpim-bonn.mpg.de}
\subjclass[2010]{35J61, 35H20}
\begin{abstract} We study travelling wave solutions, that is, solutions of the form $v(t, x) = e^{i\lambda t}u(g(t)x)$, to nonlinear Schr{\"o}dinger and Klein-Gordon equations on Riemannian manifolds, both compact and non-compact ones, with emphasis on the NLKG. Here $g(t)$ represents a one-parameter family of isometries generated by a Killing field $X$ and a case of particular interest is when $X$ has length $\leq 1$, which leads in certain settings to hypoelliptic operators with loss of at least one derivative. In the compact case, we establish existence of travelling wave solutions via ``energy'' minimization methods and prove that at least compact isotropic manifolds have \emph{genuinely} travelling waves. We establish certain sharp regularity estimates on low dimensional spheres that improve results in ~\cite{T1} and carry out the subelliptic analysis for NLKG on spheres of higher dimensions utilizing their homogeneous coset space properties. Such subelliptic phenomenon have no parallel in the setting of flat spaces. We will also study related phenomenon on complete noncompact manifolds with certain symmetry assumptions using concentration-compactness type arguments.
\end{abstract}
\maketitle
\section{\bf Introduction, Setting and Notations} Let us consider a complete Riemannian manifold $M$. Let $X$ be a Killing field on the manifold, which flows by a one-parameter family of isometries $g(t)$ of $M$. 
The following is the nonlinear Schr{\"o}dinger (NLS) equation:
\begin{eqnarray}\label{NLS}
i\partial_t v + \Delta v = - K|v|^{p - 1}v,
\end{eqnarray}
and the following is the nonlinear Klein-Gordon (NLKG) equation:
\begin{eqnarray}\label{NLKG}
\partial^2_t v - \Delta v + m^2 v = K|v|^{p - 1}v,
\end{eqnarray}
where in each case, $K > 0$ is a constant and $m \in \RR$.\newline
In this paper, we will investigate travelling wave solutions to both the NLS and the NLKG.
In the past, there has been a lot of investigation on travelling wave solutions to nonlinear Schr{\"o}dinger, Klein-Gordon and sine-Gordon equations. However, most of the literature focuses on travelling waves in an Euclidean setting: $(x, t) \in \RR \times [0, \infty)$ and their associated stability analysis. For example, see ~\cite{JMMP}, ~\cite{MJS}. In the setting $M = \RR^n$ and $g(t)x = x + tv$ for $x, v \in \RR^n$, such travelling waves have been studied in ~\cite{S} and ~\cite{BL}. \newline
As far as non-Euclidean settings are concerned, we must also mention recent interest in standing wave solutions (solutions of the form $v(t, x) = e^{i\lambda t}u(x)$) to (\ref{NLS}) and (\ref{NLKG}) in non-Euclidean settings. For example, see ~\cite{MS}, ~\cite{CM}, and ~\cite{CMMT}. To the best of our knowledge, the study of travelling waves on Riemannian manifolds was initiated in ~\cite{T1}. Our aim in this paper is to extend and build on the investigation started in ~\cite{T1}, using variants and modifications of techniques introduced in the aforementioned references, particularly ~\cite{MT} and ~\cite{CMMT}. We should also mention that the study in ~\cite{T1} focuses solely on compact manifolds. In this paper, we extend the investigation to select non-compact manifolds with certain symmetry assumptions. 

\subsection{Setting up the auxiliary equations and standing assumptions}
First, to fix notations, we define 
\begin{eqnarray}\label{Flx}
F_{\lambda, X} (u) = (-\Delta u - iXu + \lambda u, u),
\end{eqnarray}
\begin{eqnarray}\label{Fmlx}
F_{m, \lambda, X} (u) = (-\Delta u + X^2 u + 2i\lambda Xu + (m^2 - \lambda^2) u, u),
\end{eqnarray}
\begin{eqnarray}\label{Elx} 
E_{X}(u) = \frac{1}{2}(-\Delta u - iXu, u) - \frac{1}{p + 1}\int_M |u|^{p + 1}dM,
\end{eqnarray}
and
\begin{eqnarray}\label{Elx1} 
\mathcal{E}_{\lambda, X}(u) = \frac{1}{2}(-\Delta u + X^2u + 2i\lambda Xu, u) - \frac{1}{p + 1}\int_M|u|^{p + 1}dM.
\end{eqnarray}
In all of the above, and henceforth, $(u, v)$  denotes the $L^2$ inner product $(u, v) = \int_M u\overline{v} dM$.\newline
In general, if $F$ is an isometry of $M$ and we define $F^*u(x) = u(F(x))$, then it is known that the Laplacian $-\Delta$ commutes with $F^*$. Since $g(t)x$ flows by isometries, the Laplacian $-\Delta$ commutes with $g(t)^*$ for all $t$, that is,
\[
\Delta (u(g(t)x)) = (\Delta u)(g(t)x).
\] 
Using this, if we differentiate $v(t, x) = e^{i\lambda t} u(g(t)x)$ with respect to $t$, we get
\[
i\partial_t v = e^{i\lambda t}(-\lambda u(g(t)x) + iXu(g(t)x)),
\]
where, as mentioned before, $X$ is the Killing field flowing by $g(t)$. 
Thus, (\ref{NLS}) holds if and only if 
\begin{eqnarray}\label{NLSw}
-\Delta u + \lambda u - iXu = K|u|^{p - 1}u.
\end{eqnarray}
Differentiating $v(t, x) = e^{i\lambda t} u(g(t)x)$ twice with respect to $t$, we get
\[
\partial_t^2 v = e^{i\lambda t}(-\lambda^2u(g(t)x) + 2i\lambda Xu(g(t)x) + X^2u(g(t)x)).
\]
Thus, (\ref{NLKG}) holds if and only if 
\begin{eqnarray}\label{NLKGw}
-\Delta u + X^2 u +  2i\lambda Xu + (m^2 - \lambda^2) u = K|u|^{p - 1}u.
\end{eqnarray}
As we mentioned before, we assume that the Killing field $X$ is bounded, that is,
\begin{eqnarray}\label{Xb}
\langle X, X\rangle \leq b^2 < \infty, \text{   }b \in \RR.
\end{eqnarray}
On a complete manifold $M$, the Laplacian $-\Delta$ is essentially self-adjoint when defined on $C^\infty_c(M)$, and still calling $-\Delta$ the self-adjoint extension of the Laplacian, we can see that $iX$ is a small relatively bounded perturbation of $\Delta$ on which the Kato-Rellich theorem applies, which in turn 
means that $-\Delta - iX$ is self-adjoint. 
This implies\label{sym28}
\begin{eqnarray}\label{spake1}
\mbox{Spec}(-\Delta - iX) \subset [\alpha, \infty), \alpha \in \RR.
\end{eqnarray}
Now, as regards (\ref{NLKGw}), by a similar logic as above, if we assume $\langle X, X\rangle \leq b^2 < 1$, we can see that $-\Delta + X^2$ is a strongly elliptic nonnegative semidefinite self-adjoint operator and $2i\lambda X$ is a relatively bounded perturbation of $-\Delta + X^2$ with relative bound $0$. Let us quickly justify this. We have,
 \begin{align*}
(i\alpha Xu, i\alpha Xu) & = 2\alpha^2(iXu, iXu) - \alpha^2(iXu, iXu) = 2\alpha^2(X.\nabla u, X.\nabla u) + \alpha^2 (X^2 u, u)\\
& \leq 2\alpha^2 b^2 (-\Delta u, u) + \alpha^2 (X^2 u, u) \leq \alpha^2 C(-\Delta u, u) + \alpha^2 C(X^2 u, u)\\
& = \alpha^2 C ((-\Delta + X^2)u, u) \leq \alpha^2 C \V (-\Delta + X^2)u\V \V u\V \\
& \leq \alpha^2CC'\V (-\Delta + X^2)u\V^2 + \alpha^2C\frac{1}{C'}\V u\V^2.
\end{align*}
This means that $-\Delta + X^2 + 2i\lambda X$ is self-adjoint, giving 
\begin{eqnarray}\label{spake2}
\mbox{Spec}(-\Delta + X^2 + 2i\lambda X) \subset [\beta(\lambda), \infty), \text{  } \beta(\lambda) \in \RR.
\end{eqnarray}
As long as we are concentrating on compact manifolds, (\ref{Xb}) is not a geometric restriction. We will also find the opportunity to say something about non-compact manifolds which have such bounded Killing fields later. Note, however, that all non-compact manifolds do not have to have bounded Killing fields. For example, rotate the parabola $y = x^2, z = 0$ about the $y$-axis in $\RR^3$. The only Killing fields of the resulting surface of revolution generate rotations about the $y$-axis and are not bounded. The hyperbolic space $\mathbb{H}^n$ provides another example of a non-compact manifold which has no bounded Killing fields. 
\begin{remark}
Comparing the auxiliary equations (\ref{NLSw}) and (\ref{NLKGw}) we can now claim and justify a bias in our investigations towards the NLKG, which, as far as travelling waves are concerned, is harder to study because of the presence of the second order operator $X^2$ in (\ref{NLKGw}). Depending on the length of $X$, $-\Delta + X^2$ may be elliptic, subelliptic \footnote{A self-adjoint second order differential operator $\mathcal{L}$ is called subelliptic of order $\varepsilon$ $(0 < \varepsilon < 1)$ at $x \in M$ is there is a neighborhood $U$ of $x$ such that
\[
\V u\V^2_{H^\varepsilon} \leq C(|(\mathcal{L}u, u)| + \V u\V^2) \mbox{    }\forall\mbox{    } u \in C^\infty_0(U).
\]
See [F] for more details.}
or even hyperbolic. As an example, consider $\Delta = \pa^2_{x_1} +...+ \pa^2_{x_n}$ on the torus $\mathbb{T}^n$ and $X = \sqrt{2}\pa_{x_1}$. Then, $-\Delta + X^2 = \pa^2_{x_1} - \pa^2_{x_2} - ..... - \pa^2_{x_n}$. This demarcates a big point of deviation from the general methodology of ~\cite{CMMT} and ~\cite{MT}, wherein the auxiliary equations derived from (\ref{NLS}) and (\ref{NLKG}) by assuming standing wave solutions are always elliptic. 
It is also worthwhile to mention here that from the point of view of standing wave solutions, the auxiliary equations for both NLS and NLKG are absolutely similar, so there is no difference in the respective analyses. 
\end{remark}

\subsection{Outline of the paper}
Now we give a broad outline of the rest of this paper. As advertised before, we will study travelling wave solutions with bounded Killing fields $X$. ~\cite{T1} established the existence of such travelling wave solutions for (\ref{NLS}) and (\ref{NLKG}) on compact manifolds, by establishing the existence of minimizers of $F_{\lambda, X} (u)$ and $F_{m, \lambda, X} (u)$ respectively in the space $H^1(M)$ keeping the integral $\int_M |u|^{p + 1}dM$ constant. In Section \ref{EEM}, we establish the existence of constrained energy  minimizers, i.e., we minimize the energies $E_X(u)$ and $\mathcal{E}_{\lambda, X}(u)$ subject to the mass $\V u\V^2_{L^2}$ being constant and use usual variational arguments to show that these constrained minimizers actually give travelling wave solutions to (\ref{NLS}) and (\ref{NLKG}). These are respectively Proposition \ref{1.3} and Lemma \ref{1.2}.\newline
As the whole point of this investigation is to get travelling wave solutions, we must address the concern that the constrained\label{sym4} minimizers $u$ might not always satisfy $Xu \neq 0$. This is a legitimate concern, as constrained minimizers can even turn out to be constants. This concern is taken up in Section \ref{NS}, where it is shown that on fairly general spaces (compact isotropic manifolds) and for at least a non-empty set of parameters $\lambda$ and $m$, we have honest travelling wave solutions to (\ref{NLKGw}). To be precise, these are Theorem \ref{twim} (which generalizes Lemma 2.1 in ~\cite{T1}) and Theorem \ref{twim1}.\newline
So far, we have investigated (\ref{NLKGw}) under the assumption $\langle X, X\rangle < 1$. In Section \ref{SPS}, we first allow the length of the Killing field to equal $1$, albeit on a measure zero subset of $M$. Our setting is $M = S^n$ and $X = (x_1\pa_2 - x_2\pa_1)|_{S^n}$ is the restriction of a rotation vector field from $\RR^{n + 1}$. Here we extend the analysis on $S^2$ done in ~\cite{T1} to $S^n$ along somewhat similar lines of reasoning, with certain modifications, of course. The existence result for (\ref{NLKGw}) on $S^n$ is recorded in Proposition \ref{Pro}. We also improve on an estimate on $S^2$ (see (\ref{aaroaaro}) below) given in ~\cite{T1} and show that our estimate is sharp. These estimates are recorded in Subsection \ref{Mokhom} and Lemma \ref{Pro-op}.\newline
In Section \ref{CP}, we consider a particular situation where the Killing field $X$ has length $1$ everywhere. We investigate the resulting subLaplacian $-\Delta + X^2$ in the setting $M = S^7$ with reference to the contact structure available on $S^7$. The existence result for (\ref{NLKGw}) is given by Proposition \ref{esseven}. For the optimal value of $q_*$ mentioned in Proposition \ref{esseven}, refer to Subsection \ref{sevenop}, particularly Lemma \ref{sevenop1}. From the method of proof of Lemma \ref{sevenop1} it can be inferred as a corollary that the estimate in Proposition 4.2 in ~\cite{T1} is sharp.\newline
In Section \ref{NCR}, we establish our main theorems for this paper: existence of constrained $F_{m, \lambda, X}$ minimizers for (\ref{NLKGw}) and constrained $\Cal{E}_{\lambda, X}$ minimizers for (\ref{NLKGw}) in the non-compact setting. 
These are respectively Theorem \ref{100} and Theorem \ref{1000}. Let us note here that among the two, the latter is somewhat more analytically involved and requires the application of the concentration-compactness principle and a stronger symmetry assumption on the manifold to work. \newline
Finally, in Section \ref{P}, we raise the following question: for the class of non-compact manifolds considered in the statement of Theorem \ref{100}, does perturbing the Killing field slightly perturb the resulting constrained minimizer only slightly? Under a certain interpretation of this question, our answer is affirmative, and is recorded in Proposition \ref{SP}. 

\section{\bf Existence of constrained Energy minimizers on a compact manifold}\label{EEM}

 In ~\cite{T1}, it was proved that on compact $M$, with $\alpha$ as in (\ref{spake1}) and 
\beq\label{spec1}
\lambda > -\alpha,
\eeq
we have\label{sym18}
\begin{eqnarray}\label{12}
F_{\lambda, X}(u) \cong\label{sym55} \V u\V^2_{H^1} \mbox{    }\forall \mbox{    }u \in H^1(M),
\end{eqnarray}
where $H^s$\label{sym12} denotes the usual Sobolev spaces (the above fact comes from elliptic regularity once it is known that $\lambda$ is above the lowest possible eigenvalue of $-\Delta - iX$). It was also proved that with
\beq\label{spec2}
\langle X, X\rangle \leq b^2 < 1, \text{   Spec}(-\Delta + X^2 + 2i\lambda X) \subset [\beta(\lambda), \infty), \text{   }\beta(\lambda) \in \RR
\eeq and
\beq\label{spec3}
m^2 - \lambda^2 > -\beta(\lambda),
\eeq
we have
\begin{eqnarray}\label{13}
F_{m, \lambda, X} (u) \cong \V u\V^2_{H^1} \mbox{    }\forall \mbox{    }u \in H^1(M).
\end{eqnarray}
In ~\cite{T1}, (\ref{12}) was then used to minimize $F_{\lambda, X}(u)$ over $H^1(M)$, subject to the constraint 
\begin{eqnarray}\label{ic}
\int_M |u|^{p + 1} dM = \mbox{constant}.
\end{eqnarray}
Similarly, (\ref{13}) was used to minimize $F_{m, \lambda, X} (u)$ over $H^1(M)$, subject to the constraint (\ref{ic}), which would then give a solution to (\ref{NLKGw}). 
Here, we find solutions to (\ref{NLSw}) and (\ref{NLKGw}) via constrained energy minimizers, which goes as follows:\newline
For the NLS, we will try to minimize the energy $E_X(u)$ and for the NLKG, we will try to minimize the energy $\Cal{E}_{\lambda, X}$
subject\label{defeq} to 
\beq\label{MASS}
Q(u) := \V u\V^2_{L^2} = \beta \text{   (constant}).
\eeq \newline
The reason for doing this, as mentioned before, is the following:
\begin{lemma}\label{1.2}
({\bf Energy minimizers imply solutions}) Let $M$ be a compact manifold. Then
\begin{itemize}
\item If $u \in H^1(M)$ minimizes $E_X(u)$, subject to keeping the mass $\V u\V^2_{L^2} = \beta$ (constant), then $u$ solves (\ref{NLSw}) with $K > 0$ and for some $\lambda \in \RR$.
\item If $u \in H^1(M)$ minimizes  
$
\mathcal{E}_{\lambda, X}(u)$ subject to keeping the mass $\V u\V^2_{L^2} = \beta$ (constant), then $u$ solves (\ref{NLKGw}) with $K > 0$ and for some $m \in \RR$.
\end{itemize}
\end{lemma}
\begin{proof}
On calculation, we can see that with $u, v \in H^1(M)$,
\begin{eqnarray}
\frac{d}{d\tau}\bigg|_{\tau = 0}E_X(u + \tau v) = \mbox{Re}(-\Delta u - iXu - |u|^{p - 1}u, v).
\end{eqnarray}
Also, 
\begin{eqnarray}
\frac{d}{d\tau}\bigg|_{\tau = 0}Q(u + \tau v) = 2\mbox{Re}(u, v).
\end{eqnarray}
So, for the NLS, if $u \in H^1(M)$ minimizes $E_X$ constrained by $Q(u) =$ constant, then, 
\begin{eqnarray*}
v \in H^1(M), \mbox{Re}(u, v) = 0 \implies \mbox{Re}(-\Delta u - iXu - |u|^{p - 1}u, v) = 0.
\end{eqnarray*}
Since \label{sym32}$\text{Re}(. ,.)$ is a non-degenerate $\mathbb{R}$-bilinear dual pairing of $H^1(M)$ and $H^{-1}(M)$ (which is the dual of $H^1(M)$ with respect to the $L^2$ norm on $H^1(M)$), we have that there exists a $\lambda \in \mathbb{R}$ such that a mass-constrained $E_X$-minimizer $u$ satisfies
\begin{equation}\label{correc1}
-\Delta u + \lambda u - iXu = |u|^{p - 1}u.
\end{equation} 
Now, if $u$ solves (\ref{correc1}), then $u_a = au$ solves 
\begin{equation}
-\Delta u_a + \lambda u_a - iXu_a = |a|^{1 - p}|u_a|^{p - 1}u_a,
\end{equation}
which finally means that we can solve (\ref{NLSw}) for any $K > 0$.\newline
Similarly, for the NLKG, we have
\begin{equation}
\frac{d}{d\tau}\bigg|_{\tau = 0}\mathcal{E}_{\lambda, X}(u + \tau v) = \mbox{Re}(-\Delta u + 2i\lambda Xu + X^2u - |u|^{p - 1}u, v),
\end{equation}
and 
\begin{eqnarray}
\frac{d}{d\tau}\bigg|_{\tau = 0}Q(u + \tau v) = 2\mbox{Re}(u, v).
\end{eqnarray}
As before, since $\text{Re}(. ,.)$ is a non-degenerate $\mathbb{R}$-bilinear dual pairing of $H^1(M)$ and $H^{-1}(M)$, we have that there exists a $\sigma \in \mathbb{R}$ such that a mass-constrained $E_{\lambda, X}$-minimizer $u$ satisfies
\begin{eqnarray}
-\Delta u + Xu + 2i\lambda Xu + \sigma u = |u|^{p - 1}u.
\end{eqnarray}
Clearly, there exists $m \in \RR$ be such that $m^2 - \lambda^2 = \sigma$. Finally, using the scaling $u_a = au$, we see that we can produce a solution to (\ref{NLKGw}) for any constant $K > 0$ and some $m \in \RR$.
\end{proof}
So far we have argued that mass constrained energy minimizers, if they exist, would indeed give solutions to (\ref{NLSw}) and (\ref{NLKGw}). Now we have to establish the existence of such constrained energy minimizers. Let us label our assumptions
\beq\label{notun1}
\langle X, X\rangle \leq b^2,\text{   } b \in \RR,
\eeq
and 
\beq\label{notun2}
\langle X, X\rangle \leq b^2 < 1,\text{   } b \in \RR.
\eeq
\begin{proposition}\label{1.3}({\bf Existence of constrained energy minimizers})
On a compact Riemannian manifold $M$, if $p \in (1, 1 + 4/n)$,  then we can find, assuming (\ref{notun1}) and (\ref{notun2}) respectively, minimizers for $E_X$, and $\mathcal{E}_{\lambda, X} $ for all $ \lambda \in \RR$, when the minimization is done in the class of $H^1(M)$ functions having constant $L^2$-norm.
\end{proposition}
\begin{proof}
Let us define
\begin{eqnarray}
I_\beta =  \mbox{inf}\{E_X | u \in H^1(M), Q(u) =  \beta\},
\end{eqnarray}
\begin{eqnarray}
I'_\beta =  \mbox{inf}\{\mathcal{E}_{\lambda, X} | u \in H^1(M), Q(u) =  \beta\}.
\end{eqnarray}
\\
Recall the Gagliardo-Nirenberg inequality:
\begin{eqnarray}\label{JALATON1}
\V u\V_{L^{p + 1}} \leq C\V u\V^{1 - \gamma}_{L^2}\V u\V^{\gamma}_{H^1},
\end{eqnarray}
where $\gamma = \frac{n}{2} - \frac{n}{p + 1}$, and hence $\gamma (p + 1) <2$.\newline
\\
Choosing $\lambda$ satisfying (\ref{spec1}), we have, 
\begin{align*}
F_{\lambda, X} (u) & = (-\Delta u - iXu + \lambda u, u) = (-\Delta u - iXu, u) + (\lambda u, u) \\
& = (-\Delta u - iXu, u) - \frac{2}{p + 1}\int_M|u|^{p + 1}dM + \frac{2}{p + 1}\int_M|u|^{p + 1}dM + (\lambda u, u)\\
& = 2E_X(u) + \frac{2}{p + 1}\int_M|u|^{p + 1}dM + \lambda Q(u).
\end{align*}
This gives via (\ref{JALATON1}),
\beq\label{broken}
F_{\lambda, X}(u) \leq 2E_X(u) + CQ(u)^{(p + 1)\frac{(1 - \gamma)}{2}}\V u\V^{\gamma(p + 1)}_{H^1} + \lambda Q(u), \text{   }C > 0.
\eeq
\label{sym2}
This derivation implies two things: \newline
Since $Q(u) = \beta$ is constant, $I_\beta > -\infty$, since $F_{\lambda, X}(u) \geq 0$. Also, since $\gamma(p + 1) < 2$, if $u_\nu$ is a sequence in $H^1(M)$ such that $E_X(u_\nu) \rightarrow I_\beta$, then (\ref{broken}) implies that $\V u_\nu\V_{H^1}$ remains bounded. This is because, $F_{\lambda, X}(u) \cong \V u\V^2_{H^1}$.\newline
Similarly, for the NLKG, choosing $m$ such that $m^2 - \lambda^2 > -\beta(\lambda)$, with $\beta(\lambda)$ defined as in (\ref{spec2}), we have
\begin{align*}
\V u\V^2_{H^1} & \cong F_{m, \lambda, X}(u) \\
& = (-\Delta u + X^2u + 2i\lambda Xu + ((m^2 - \lambda^2)u, u)\\
& = (-\Delta u + X^2u + 2i\lambda Xu, u) - \frac{2}{p + 1}\int_M|u|^{p + 1}dM + \frac{2}{p + 1}\int_M|u|^{p + 1}dM \\
& + ((m^2 - \lambda^2)u, u)\\
& = 2\mathcal{E}_{\lambda, X}(u) + \frac{2}{p + 1}\int_M|u|^{p + 1}dM + (m^2 - \lambda^2)Q(u).
\end{align*}
This gives
\begin{align}\label{broken1}
\V u\V^2_{H^1} & \lesssim 2\mathcal{E}_{\lambda, X}(u) + CQ(u)^{\frac{p + 1}{2}(1 - \gamma)}\V u\V^{\gamma(p + 1)}_{H^1} + (m^2 - \lambda^2)Q(u)\\
& = 2\mathcal{E}_{\lambda, X}(u) + K\V u\V^{\gamma(p + 1)}_{H^1} + K',
\end{align}
where $K, K' > 0$ are constants. So, as before, $I'_\beta > -\infty$ and if $u_\nu \in H^1(M)$ is a sequence satisfying $\mathcal{E}_{\lambda, X}(u_\nu) \to I'_\beta $, then $\V u_\nu\V_{H^1(M)}$ must be bounded. \newline
So, in both the cases, passing to a subsequence if need be, we can assert that there exists $u \in H^1(M)$ such that
\[
u_\nu \rightarrow u
\]
 weakly in $H^1(M)$. \newline
Now, by the compactness of Sobolev embedding\label{sym17} $H^1(M) \hookrightarrow L^2(M)$, $u_\nu$ has a convergent subsequence, called $u_\nu$ again by abuse of notation, converging in $L^2$-norm, and the $L^2$-limit is $u$. So, by the triangle inequality, $\V u\V_{L^2} = \Vert u_\nu\Vert_{L^2}$. \newline
Now to prove that $u$ attains the infimum $I_\beta$, that is, 
\[
E_X(u) = I_\beta.
\]
We know that 
\[
E_X(u) = \frac{1}{2}F_{\lambda, X}(u) - \frac{1}{p + 1}\int_M |u|^{p + 1}dM - \frac{1}{2}\lambda Q(u).
\]
Since $u_\nu \to u$ in $L^{p + 1}$-norm, by the triangle inequality, we have $\V u_\nu\V_{L^{p + 1}} \rightarrow \V u\V_{L^{p + 1}}$. So it suffices to establish that 
\[
F_{\lambda, X}(u) \leq \liminf F_{\lambda, X}(u_\nu).
\]
But this is a consequence of the fact that $u_\nu \to u$ weakly in $H^1$ and $F_{\lambda, X}(u) \cong \V u\V^2_{H^1}$. This settles the case for the NLS.\newline
For the NLKG, we have to prove that $\Cal{E}_{\lambda, X}(u) = I'_\beta$. Now, 
\beq\label{EF}
\Cal{E}_{\lambda, X}(u) = \frac{1}{2}F_{m, \lambda, X}(u) - \frac{1}{p + 1}\int_M |u|^{p + 1}dM - \frac{1}{2}(m^2 - \lambda^2) Q(u).
\eeq
Since $\V u_\nu\V_{L^2} = \V u\V_{L^2}$ and $ \V u_\nu\V_{L^{p + 1}} \to \V u\V_{L^{p + 1}}$, we just have to argue that 
\[
F_{m, \lambda, X}(u) \leq \liminf F_{m, \lambda, X}(u_\nu).
\]
As argued before, this derives from the facts that $u_\nu \to u$ weakly in $H^1(M)$ and $F_{m, \lambda, X}(u) \cong \V u\V^2_{H^1}$. That finishes the proof.
\end{proof}
\noindent {\bf Remark:} Note that the constrained $F_{\lambda, X}$ or $F_{m, \lambda, X}$ minimizers give solutions to (\ref{NLSw}) and (\ref{NLKGw}) respectively for $p \in (1, \frac{n + 2}{n - 2})$ (the optimal range for compact Sobolev embeddings), while the constrained $E_X$ or $\Cal{E}_{\lambda, X}$ minimizers give solutions to (\ref{NLSw}) and (\ref{NLKGw}) respectively for a smaller range of $p$; to wit, $p \in (1, 1 + 4/n)$. However, it is not apriori clear that the solutions obtained from the two minimization schemes are the same. Since they are different variational formulations, they can potentially give different solutions.

\section{\bf Nontriviality of solutions and a few other remarks}\label{NS}

 As mentioned in the outline, we must note that the mere existence of minimizers will not guarantee waves that are {\em actually} travelling. For example, on a compact manifold $M$ 
\[
u = [(m^2 - \lambda^2)/K]^{\frac{1}{p - 1}}
\]
solves (\ref{NLKGw}) and it is natural to ask if this is an $F_{m, \lambda, X}$ minimizer subject to (\ref{ic}). In general, it is also possible to have non-constant constrained minimizers $u$ such that $Xu = 0$; such waves will definitely not be travelling.
~\cite{T1} established the following result in this regard:
\begin{proposition}
Given $n \geq 2$, $p \in (1, (n + 2)/(n - 2))$, $m > 0, K > 0$, and a Killing field on $S^n$ such that $\langle X, X\rangle \leq b^2 < 1$, there exists $\varepsilon_0 > 0$ such that, for $\varepsilon \in (0, \varepsilon_0]$, the constrained $F_{m, 0, X}$-minimization process produces a genuinely travelling wave solution to 
\[
-\Delta u + X^2u + \frac{1}{\varepsilon^2}m^2u = \frac{1}{\varepsilon^2}K|u|^{p - 1}u.\] 
\end{proposition}
Here we extend the above result to arbitrary compact connected isotropic manifolds. To begin the discussion, we first quote the following
\begin{lemma}({\bf Global constrained minimizer of $(-\Delta u + m^2u, u)_{L^2(\mathbb{R}^n)}$})\label{le}\\
Given
\beq
n \geq 2, \mbox{    } p \in (1, \frac{n + 2}{n - 2}), \mbox{    } A \in (0, \infty),
\eeq
there is a minimizer $u^0 \in H^1(\mathbb{R}^n)$ to $F_m(u) = ((-\Delta + m^2)u, u)_{L^2(\mathbb{R}^n)}$ subject to the constraint $\int_{\RR^n}|u|^{p + 1}d\RR^n = A$.
\end{lemma}
\begin{proof} For a proof, refer to Lemma 2.2 of ~\cite{T1} and also ~\cite{BL}.\end{proof}
We just want to point out the following important fact about the above lemma: the proof, as stated in ~\cite{T1} and ~\cite{BL}, also establishes that we can arrange so that the constrained minimizer $u^0$ is a radial function. We will use this fact in the sequel.
Now, we have the following
\begin{theorem} \label{twim}({\bf travelling waves on isotropic  manifolds})
Given a compact connected isotropic manifold\footnote{Isotropic manifolds are defined as those Riemannian manifolds such that, given any $p \in M$ and unit vectors $v, w \in T_p(M)$, there exists $\varphi \in $ Isom$(M)$ such that $\varphi(p) = p$ and $d\varphi_p (v) = w$.} $M$ of dimension $n \geq 2, p \in (1, (n + 2)/(n - 2))$, $m > 0, K > 0$ and a Killing field $X$ such that $\langle X, X\rangle \leq b^2 < 1$, there exists $\delta > 0$ such that for $\varepsilon \in (0, \delta]$, the constrained $F_{m, 0, X}$-minimizing process produces a solution to 
\[
-\Delta u + X^2 u + \frac{1}{\varepsilon^2}m^2u = \frac{1}{\varepsilon^2}K|u|^{p - 1}u 
\]
with $Xu \neq 0$.
\end{theorem}
\begin{proof}
We have
\[
F_{m, 0, 0}(u) = ((-\Delta + m^2)u, u) = \V\nabla u\V^2_{L^2} + m^2\V u\V^2_{L^2}.
\]  
Then, for all $u \in H^1(M)$, we have
\[
F_{m, 0, X}(u) = \V\nabla u\V^2_{L^2} + (X^2u, u) + m^2\V u\V^2_{L^2} = \V\nabla u\V^2_{L^2} - \V Xu\V^2_{L^2} + m^2\V u\V^2_{L^2} \leq F_{m, 0, 0}(u).
\]
Now, if $u$ is not travelling, that is, $Xu = 0$, then $F_{m, 0, X}(u) = F_{m, 0, 0}(u)$, which means that if $u \in H^1(M)$ minimizes $F_{m, 0, X}$ subject to (\ref{ic}), then $u$ also minimizes $F_{m, 0, 0}$ subject to (\ref{ic}). Now let us consider the function $v(x) = u(\phi(x))$, where $\phi \in$ Isom$(M)$\label{sym24}. We have $F_{m, 0, 0}(v) = F_{m, 0, 0}(u)$. Also,
\begin{align*}
F_{m, 0, X}(v) & \geq F_{m, 0, X}(u) \text{   (since  } u \text{  is a   } F_{m, 0, X}-\text{minimizer})\\
& = F_{m , 0, 0}(u) \text{   (since   } Xu = 0).
\end{align*}
Now, 
\begin{align*}
F_{m, 0, 0}(v) & = (-\Delta u(\phi(x)) + m^2u(\phi(x)), u(\phi(x)))\\
& = \int_M (-\Delta u(\phi(x)) + m^2u(\phi(x))\overline{u(\phi(x))}dM\\
& = \int_M (-\Delta u(y) + m^2u(y))\overline{u(y)} dM\\
& = F_{m, 0, 0}(u).
\end{align*}
Now, we have 
\[
F_{m, 0, X}(v) = F_{m, 0, 0}(v) \implies Xv = 0.
\]
This happens for all $\phi \in$ Isom$(M)$. Now, choose a point $P \in M$ and let $Y$ be a smooth vector field on $M$ such that\label{XP} $X_P$ and $Y_P$ are linearly independent and $X_P$ and $Y_P$ have the same length. Consider the isometry $\phi \in \text{Isom}(M)$ such that $\phi (P) = P$ and $d\phi (X_P) = Y_P$. Then, if $v(x) = u(\phi (x))$, we have $Xv|_P = Yu|_P = 0$. Since this happens for all vector fields $Y$, we can see that $u$ is locally constant. Also, since $p$ can be any point on $M$ (and $M$ is connected), we finally have that $u$ is globally constant.\newline
Now, let us scale the metric on\label{sym21} $(M, g)$ to $M_r = (M, g^r_{ij})$ by $g^r_{ij} = r^2g_{ij}$. Consider a metric ball $U$ of radius $k$ on $M$ which is small enough so that $U$ is diffeomorphic to the open Euclidean 1-ball in $\mathbb{R}^n$. Let $U^r$ be the dilated image of $U$ under the scaling. On $M_r$, consider the vector field $X_r = \frac{1}{r}X$. Let $u_r$ denote the minimizer of $F^r_{m, 0, X_r}(u)$, subject to $\int_{M_r} |u|^{p + 1}dM_r = A$. If $u_r$ is constant, on calculation,
\[
u_r = (\frac{A}{V})^{\frac{1}{p + 1}}r^{-\frac{n}{p + 1}},
\]
where $V$ is the volume of $(M, g)$.\newline
That gives,
\begin{align*}
F^r_{m, 0, X_r}(u_r) & = m^2(\frac{A}{V})^{\frac{2}{p + 1}}Vr^nr^{-\frac{2n}{p + 1}}\\
& = Cr^{\frac{n(p - 1)}{p + 1}},
\end{align*}
where $C$ is a constant. Since $X_ru_r = 0$, this is also the infimum of $F^r_{m, 0, 0}(u)$, subject to $\int_{M_r} |u|^{p + 1}dM_r = A$.\newline
Now,\label{sym27}
\[
\inf_{u \in H^1(M_r), \text{supp } u \subset U^r} F^r_{m, 0, 0}(u) \geq \inf_{u \in H^1(M_r)} F^r_{m, 0, 0}(u).\]
So, as $r \to \infty$, 
\begin{equation}\label{tend}
\inf_{u \in H^1_0(U^r)} F^r_{m, 0, 0}(u) \to \infty.
\end{equation}
But, as $r \to \infty$, $U^r$ approaches the flat Euclidean space $\RR^n$.\newline 
Let $P$ be the centre of the balls $U^r$, which have radius $rk$. Using the radial minimizer\label{sym33} $u^0$ of Lemma \ref{le}, define
\begin{equation}
v_r(x) = \chi(x)u^0(\text{dist}_r(P, x)), x \in U^r,
\end{equation} 
where $\text{dist}_r$ is the metric distance in $(M, g^r_{ij})$ and $\chi(x)$ is a smooth radial cut-off function such that $\chi \equiv 1$ on $B^r_P(rk - \frac{1}{r})$, where the superscript $r$ on the ball denotes the ball in the $g^r_{ij}$ metric.\newline
We have 
\[
\int_{M_r} |v_r|^{p + 1} \to A\]
and \[ F^r_{m, 0, 0}(v_r) \to F^0_{m, 0, 0}(u^0) = ((-\Delta_{\RR^n} + m^2)u, u)_{L^2(\mathbb{R}^n)}< \infty,\] thereby contradicting (\ref{tend}).\newline
So, for $r$ large, there exists a constrained minimizer $u_r$ such that $X_ru_r \neq 0$, which solves
\begin{equation}\label{898989}
-\Delta_r u_r + X^2_ru_r + m^2u_r = K|u_r|^{p - 1}u_r,
\end{equation}
where $K > 0$ is arbitrary, as we can scale $u_r \mapsto au_r$. Seeing that $-\Delta_r = -\frac{1}{r^2}\Delta$ and $X_r = \frac{1}{r}X$, and scaling back (\ref{898989}), we finally have our result.

\end{proof}

The next natural question is: what about the nontriviality of energy minimizers? Do they necessarily {\em need} to be travelling for at least a certain set of parameters? As before, we have the following:
\begin{theorem}\label{twim1} 
Given a compact connected isotropic manifold $M$ of dimension $n \geq 2$, $p \in (1, 1 + 4/n)$, $m > 0, K > 0$ and a Killing field $X$ such that $\langle X, X\rangle \leq b^2 < 1$, there exists $\delta > 0$ such that for $\varepsilon \in (0, \delta]$, the mass constrained $\Cal{E}_{ 0, X}$-minimizing process produces a solution to 
\[
-\Delta u + X^2 u + \frac{1}{\varepsilon^2}m^2u = \frac{1}{\varepsilon^2}K|u|^{p - 1}u 
\]
with $Xu \neq 0$.
\end{theorem}
\begin{proof}
Proceeding as in the proof of Theorem \ref{twim}, we see that if $u$ minimizes $\Cal{E}_{0, X} = \frac{1}{2}\V \nabla u\V^2 - \frac{1}{2}\V Xu\V^2 -  \frac{1}{p + 1}\int_M |u|^{p + 1}dM$ subject to 
\beq\label{ic2}
\V u\V^2_{L^2} = \beta,
\eeq
and $Xu = 0$, then $u$ is constant.\newline
Now, let us scale the metric on $(M, g)$ to $M_r = (M, g^r_{ij})$ by $g^r_{ij} = r^2g_{ij}$. Consider a metric ball $U$ of radius $k$ on $M$ which is small enough so that $U$ is diffeomorphic to the open Euclidean 1-ball in $\mathbb{R}^n$. Let $U^r$ be the dilated image of $U$ under the scaling. On $M_r$, consider the vector field $X_r = \frac{1}{r}X$. Let $u_r$ denote the minimizer of $\Cal{E}^r_{0, X_r}$ subject to (\ref{ic2}). Then, on calculation, $u_r = (\frac{\beta}{V})^{1/2}r^{-n/2}$, where $V$ is the volume of $(M, g)$, and $\Cal{E}^r_{0, X_r}(u_r) = - \frac{1}{p + 1}(\frac{\beta}{V})^{\frac{p + 1}{2}}Vr^{n - \frac{n(p + 1)}{2}}$. This gives, 
\beq
\Cal{E}^r_{0, X_r}(u_r) \to 0 \text{   as   } r \to \infty.
\eeq
If for all $r$, $X_r u_r = 0$, then this also means
\beq\label{arkoto}
\Cal{E}^r_{0, 0}(u_r) \to 0 \text{   as   } r \to \infty.
\eeq
Now, we have, 
\[
\inf_{u \in H^1(M), \text{supp }u \subset U^r}\Cal{E}^r_{0, 0}(u) \geq \inf_{u \in H^1(M)}\Cal{E}^r_{0, 0}(u),
\]
which means,
\beq
\liminf_r \inf_{u \in H^1(M), \text{supp }u \subset U^r}\Cal{E} ^r_{0, 0}(u) \geq 0.
\eeq
But, as $r \to \infty$, $U^r$ approaches the Euclidean space $\RR^n$, and 
\beq\label{nt}
\inf_{u \in H^1(\RR^n)}\Cal{E}_{0, 0}(u) < 0.
\eeq
For a proof of the nontrivial claim (\ref{nt}), see Appendix A.2 of ~\cite{CMMT}. This gives a contradiction. By corresponding scaling arguments outlined in the proof of Theorem \ref{twim}, this proves our contention.
\end{proof}
\section{\bf $\langle X, X\rangle \leq 1$: subelliptic phenomenon on $S^n$, $n \geq 3$.}\label{SPS}
 In this section, let us relax (for the travelling waves of the NLKG) slightly the previously held restriction that $\langle X, X\rangle \leq b^2 < 1$. Now, if we allow the length of $X$ to equal 1 at some points of $M$, then $-\Delta + X^2$ is not elliptic there anymore, which somewhat restricts the techniques we have at our disposal. To balance for that, we will carry out the investigation on a much more restricted geometric setting, namely, the sphere $S^n$. ~\cite{T1} has a detailed investigation of this on the sphere $S^2$, which we recall below. If $S^2$ is centred at the origin in $\RR^3$, and $X, Y$ and $Z$ denote $2\pi$-periodic rotations of $S^2$ about the $x, y$ and $z$-axis respectively, and $L_0 = \Delta - X^2$, we have the following
\begin{proposition}(~\cite{T1})
\beq\label{aaroaaro}
\Cal{D}((-L_0)^{1/2}) \hookrightarrow L^q(S^2), \text{   } \forall q \in [2, 6),
\eeq
and the inclusion is compact. Then, under the assumptions 
\[
2 < p + 1 < 6,
\]
and \[|\lambda| < \frac{1}{2}, m^2 > \lambda^2,\]
given $K > 0$, (\ref{NLKGw}) has a nonzero solution $u \in \Cal{D}((-L_0)^{1/2})$.
\end{proposition}
We will now extend the above analysis to a sphere of dimension $n$. We will also improve on the estimate (\ref{aaroaaro}) and show that our estimate is sharp. 
\subsection{Setting up the problem}\label{cris}
 Let $X_{ij}, i < j$, denote the vector field on $S^n$, which 
is the restriction of the vector field $x_i\partial_j - x_j\partial_i$ on $\mathbb{R}^{n + 1}$ onto $S^n$. It is known that the Laplacian on $S^n$ is given by $\Delta = \displaystyle\Sigma_{i < j} X_{ij}^2$ (actually, much more general statements can be made; for a  survey article, see ~\cite{C}, Appendix B.4).
 So pick one of these $X_{ij}$'s, say without loss of generality, $X_{12}$, henceforth called just $X$. It is to be noted that $\langle X, X\rangle < 1$ does not hold here always. So $L_0 = \Delta - X^2$ is not globally elliptic, but it satisfies H{\"o}rmander's condition for hypoellipticity. Let us justify this: from the above sum of squares, we see that $L_0$ is elliptic on $S^n \setminus \{(\pm 1, 0,...,0), (0, \pm 1,0,..., 0)\}$. Pick, without loss of generality, the point $(1, 0,..., 0)$. 
On calculation, 
\[
[X_{23}, X_{13}]  = (x_2\partial_1 - x_1\partial_2)|_{S^n} = X_{12}.
\]
Together, at $(1, 0,..,0)$, $X_{1j}$, $j = 2, 3,..., n + 1$ generate  the full tangent\label{Hypo} space of $S^n$.\newline
Also, by results in ~\cite{T2} (Chapter XV, Theorem 1.8), $L_0$ is hypoelliptic with loss of a single derivative (also see the Appendix at the end), which means the following\label{sym13}:
\begin{eqnarray}
L_0\phi \in H^s_{\text{loc}} \Rightarrow u \in H^{s+1}_{\text{loc}}.
\end{eqnarray} 
This gives that  
\[
\mathcal{D}(L_0) \subseteq H^1(S^n),
\] 
which in turn implies, by complex interpolation\label{sym5},
\beq\label{badd}
\mathcal{D}((-L_0)^{1/2}) = [L^2(S^n), \Cal{D}(-L_0)]_{1/2} \subset [L^2(S^n), H^1(S^n)]_{1/2} = H^{1/2}(S^n).
\eeq
Now\label{sym16}, if we let 
\beq L_\alpha = L_0 - i\alpha X,
\eeq
we see that $L_\alpha$ is self-adjoint for all $\alpha \in \RR$. However, to work with $(-L_\alpha)^{1/2}$, or even to define it via the spectral theorem, we need to establish the negative semidefiniteness of $L_\alpha$ for a certain range of $\alpha$ and establish what the range is. 
We actually have
\begin{lemma}
$L_\alpha = \Delta - X^2 - i\alpha X$ is negative semidefinite for $|\alpha| < n - 1$.
\end{lemma}
\begin{proof}
To start, we can do an eigenvector decomposition of $L^2(S^n)$ with respect to the self-adjoint $\Delta$. Since $X$ is Killing, it commutes with $\Delta$ and respects the eigenspace decomposition. This means, $X$ maps any eigenspace of $\Delta$ into itself. 
Let $V_k$ denote the space of degree $k$ harmonic homogeneous polynomials, defined on $\mathbb{R}^{n + 1}$ and then restricted to $S^n$. It is known that all the eigenfunctions of the Laplacian on $S^n$ are given by the members of $V_k$ (see ~\cite{T4}, Chapter 8, Section 4). The eigenvalue corresponding to $V_k$ is $k(k + n -1)$. It is also known that $V_k$ is generated by polynomials of the form \newline $P_C(x) = (c_1 x_1 + ..... + c_{n+1}x_{n+1})^k$, where $x_i \in \RR^{n + 1}$, $c_i \in \mathbb{C}$ and $\Sigma c_i^2 = 0$.  
Now, 
\[
X(P_C(x)\bigg|_{S^n}) = [(x_1\pa_2 - x_2\pa_1)P_C(x)]\bigg|_{S^n}.
\]
But,
\begin{align*}
[(x_1\pa_2 - x_2\pa_1)P_C(x)]\bigg|_{S^n} & = [(x_1\pa_2 - x_2\pa_1)(c_1x_1 +....+ c_{n + 1}x_{n + 1})^k]\bigg|_{S^n}\\
& = k(x_1c_2 - x_2c_1)(c_1x_1 +....+ c_{n + 1}x_{n + 1})^{k - 1}\bigg|_{S^n}.
\end{align*}
If $P_C(x)\bigg|_{S^n}$ is an eigenfunction of $X$, then we must have $\gamma(x_1c_2 - x_2c_1)\bigg|_{S^n} = (c_1x_1 + ... + c_{n + 1}x_{n + 1})\bigg|_{S^n}, \gamma \in \CC$. That gives, $c_3 = ... = c_{n + 1} = 0$. Also, using that $c^2_1 + c^2_2 = 0$, we see that $\gamma = \pm i$. From the above discussion, we conclude that an element $f \in V_k$ has eigenvalue $l$ with respect to $iX$, then $|l| \leq k$, and $k$ is also a possible eigenvalue. \newline
Now, on the finite dimensional vector space $V_k$, the operator $iX$ is Hermitian, allowing it to have a basis of eigenfunctions, say, $v_1, v_2,..., v_{m_k}$, where $m_k = \text{dim }V_k$. Choose any of these basis eigenfunctions, and call it $u^*$. Let the eigenvalue corresponding to $u^*$ be $l$, where, as commented before, $|l| \leq k$. Then,
\begin{align*}
(-L_{\alpha}u^*, u^*) & = k(k + n - 1) \V u^*\V^2_{L^2} + (X^2u^* + i\alpha Xu^*, u^*)\\
& \geq k(k + n - 1)\V u^*\V^2_{L^2} - l^2 \V u^*\V^2_{L^2} - |\alpha| l\V u^*\V^2_{L^2}\\
& = (k(k + n - 1) - l^2 -  |\alpha| l)\V u^*\V^2_{L^2} \geq 0.
\end{align*} 
This finally implies that $L_\alpha$ is negative semidefinite with one dimensional kernel (containing only the constants) when $|\alpha| < n - 1$. 
\end{proof}
Clearly, $\Cal{D}((-L_{2\lambda})^{1/2}) = \Cal{D}((-L_{0})^{1/2})$, and by (\ref{badd}), both lie inside $H^{1/2}(S^n) $, and by Sobolev embedding, $H^{1/2}(S^n) \hookrightarrow L^{\frac{2n}{n - 1}}(S^n)$.\newline
Now, let $q_*$ be the optimal (greatest) number such that
\beq\label{OPTI1}
\mathcal{D}((-L_0)^{1/2}) \subset L^q(S^n), \forall \mbox{   } q \in [2, q_*],
\eeq
or
\beq\label{OPTI2}
\mathcal{D}((-L_0)^{1/2}) \subset L^q(S^n), \forall \mbox{   } q \in [2, q_*).
\eeq
Whichever be the case, we will now argue that the inclusions (\ref{OPTI1}) and (\ref{OPTI2}) are continuous via the closed graph theorem applied to the inclusion operator. For the norm on $\Cal{D}((-L_{0})^{1/2})$, we use the graph norm given by
\[
\V u\V_{\Cal{D}((-L_{0})^{1/2})}^2 = ((-L_0)^{1/2} u, (-L_0)^{1/2} u) + (u, u),\]
which turns $\Cal{D}((-L_{0})^{1/2})$ into a Hilbert space (see Proposition 1.4 of ~\cite{Sc}). Let us argue the applicability of the closed graph theorem above. It suffices to demonstrate the impossibility of the following scenario: $u_n$ is a sequence in $\Cal{D}((-L_{0})^{1/2})$, such that $u_n \to u$ in $\Cal{D}((-L_{0})^{1/2})$-norm, $u_n \to v$ in $L^q$ norm, and $u \neq v$. Observe that  $u_n \to u$ in $\Cal{D}((-L_{0})^{1/2})$-norm implies that $u_n \to u$ in $L^2$-norm. Also, being in a compact setting, $u_n \to v$ in $L^q$-norm means $u_n \to v$ in $L^2$-norm, meaning $u = v$.\newline We also note that the\label{compactn} continuity of the inclusion in (\ref{OPTI1}) or (\ref{OPTI2}) will actually guarantee that (\ref{OPTI2}) is compact. Let us argue this first: by interpolation (see ~\cite{T3}, Chapter 4, Section 2), for all $q \in [2, q^*)$, we can produce $s \in (0, 1)$ such that $\Cal{D}((-L_{0})^{s/2}) \subset L^q(S^n)$ is a continuous inclusion.
\[
\Cal{D}((-L_0)^{s/2}) = [L^2(S^n), \Cal{D}((-L_0)^{1/2})]_s \subset [L^2(S^n), L^{q'}(S^n)]_s,\]
where $q' < q_*$ is chosen such that $[L^2(S^n), L^{q'}(S^n)]_s = L^q$. We can then compose the continuous inclusion $\Cal{D}((-L_{0})^{s/2}) \subset L^q(S^n)$ with the compact inclusion $\Cal{D}((-L_{0})^{1/2}) \hookrightarrow \Cal{D}((-L_{0})^{s/2})$ (the fact that this last inclusion is compact is not trivial; for a proof, see Theorem A.38 of ~\cite{MK}). 
Since the composition of a bounded and a compact operator is compact, we have our claim that continuity of the inclusion (\ref{OPTI1}) or (\ref{OPTI2}) would imply compactness of (\ref{OPTI2}).\newline
Now, we have our existence result:
\begin{proposition}\label{Pro} ({\bf Existence result on $S^n$})
With $X = x_1\pa_2 - x_2\pa_1|_{S^n}$, where $x_1\pa_2 - x_2\pa_1$ denotes the usual rotation vector field on $\RR^{n + 1}$, assume
\begin{eqnarray}
2 < p + 1 < q_*,
\end{eqnarray}
where $q_*$ is the greatest number such that (\ref{OPTI1}) or (\ref{OPTI2}) is compact.
Also assume 
\begin{eqnarray}
|\lambda| < \frac{n - 1}{2}, \text{    }m^2 > \lambda^2.
\end{eqnarray}
Then, given $K > 0$, the equation
\begin{eqnarray}
-L_{2\lambda}u + (m^2 - \lambda^2)u = K|u|^{p - 1}u
\end{eqnarray}
has a nonzero solution $u \in \mathcal{D}((-L_{2\lambda})^{1/2})$.
\end{proposition}
\begin{proof}
As we have shown above, $-L_{2\lambda}$ is positive semidefinite when $|\lambda| < \frac{n - 1}{2}$. So, the spectral theorem gives a definition of $(-L_{2\lambda})^{1/2}$. Then we use the fact that 
\begin{eqnarray}\label{GAPRE}
F_{m, \lambda, X}(u) = (-L_{2\lambda} u, u) + (m^2 - \lambda^2)(u, u) \cong \V u\V^2_{\mathcal{D}((-L_{2\lambda})^{1/2})},
\end{eqnarray}
where $\V u\V^2_{\mathcal{D}((-L_{2\lambda})^{1/2})}$ is the graph norm given by 
\[
\V u\V^2_{\mathcal{D}((-L_{2\lambda})^{1/2})} = ((-L_{2\lambda})^{1/2} u, (-L_{2\lambda})^{1/2} u) + (u, u),
\]
which turns $\mathcal{D}((-L_{2\lambda})^{1/2})$ into a Hilbert space (see Proposition 1.4 of ~\cite{Sc}). \newline
Let
\[
I_\beta = \inf \{F_{m, \lambda, X}(u) : u \in \mathcal{D}((-L_{2\lambda})^{1/2})\},
\]
under the constraint (\ref{ic}). Now, take a sequence of functions $u_\nu \in \mathcal{D}((-L_{2\lambda})^{1/2})$ such that $F_{m, \lambda, X}(u_\nu) \to I_\beta$. Then, (\ref{GAPRE}) implies that $\V u_\nu\V_{\Cal{D}((-L_{2\lambda})^{1/2})}$ is uniformly bounded, which in turn means (a subsequence of) $u_\nu$ weakly converges to $u \in \mathcal{D}((-L_{2\lambda})^{1/2})$. By virtue of the compactness of (\ref{OPTI2}), $u_\nu$ has a subsequence, still called $u_\nu$ with mild abuse of notation, that is strongly $L^{p + 1}$ convergent to $u$, where $p + 1 < q_*$, meaning that $\V u\V^{p + 1}_{L^{p + 1}} = \beta$, as in (\ref{ic}). Also, 
\[
F_{m, \lambda, X}(u) \leq \liminf F_{m, \lambda, X}(u_\nu).
\]
So, $u \in \mathcal{D}((-L_{2\lambda})^{1/2}))$ gives a constrained minimizer to $F_{m, \lambda, X}(u)$ subject to (\ref{ic}).
The constrained minimizer will give a solution to (\ref{NLKGw}), as wanted.
\end{proof}
\noindent {\bf Remark}: Arguing as before with the closed graph theorem applied to the identity map, we can establish that 
\beq 
\V .\V_{\mathcal{D}((-L_{2\lambda})^{1/2}))} \cong \V .\V_{\mathcal{D}((-L_{0})^{1/2}))}.
\eeq
This is because $u_n \to u$ in $\V .\V_{\mathcal{D}((-L_{2\lambda})^{1/2}))}$-norm implies $u_n \to u$ in $L^2$-norm, and if $u_n \to v$ in $\V .\V_{\mathcal{D}((-L_0)^{1/2}))}$-norm, then $u_n \to v$ in $L^2$-norm, whence $u = v$.\newline
It is not apriori clear that the constrained $F_{m, \lambda, X}$ minimizer obtained above is non-constant always. However, when $\lambda = 0$, the arguments of Theorem \ref{twim} go through, giving the following:
\begin{lemma} Given $p + 1 \in (2, q_*)$, $m > 0, K > 0$ and the Killing field $X$ as mentioned at the beginning of Subsection \ref{cris}, there exists $\delta > 0$ such that for $\varepsilon \in (0, \delta]$, the constrained $F_{m, 0, X}$-minimizing process on $S^n$ produces a solution to 
\[
-\Delta u + X^2 u + \frac{1}{\varepsilon^2}m^2u = \frac{1}{\varepsilon^2}K|u|^{p - 1}u 
\]
with $Xu \neq 0$.
\end{lemma}

\subsection{What is the optimal $q_*$?}\label{Mokhom}
On $S^n$, $H^{1/2}$ Sobolev embeds in $L^{\frac{2n}{n - 1}}$. By mimicking the calculations in ~\cite{T1}, we now try to see if this can be improved. For any vector field $X_{ij} \neq X$, the points where $X_{ij}$ vanish, will lie on, say, $U_{ij}$ where $U_{ij}$ is isometric to $S^{n - 2}$. By using the ellipticity of $L_0$ away from points which have coordinates $(\pm 1,0,...,0)$ and $ (0,\pm 1,..,0))$, or the ``poles'', we have, 
\begin{eqnarray}
u \in \mathcal{D}((-L_0)^{1/2}) \Rightarrow \phi u \in H^1(S^n) \subseteq L^{\frac{2n}{n - 2}}(S^n)
\end{eqnarray}
by Sobolev embedding, where $\phi \in C^\infty_c(S)$\label{sym8} and $S = S^n \setminus \{(\pm 1,0,...,0), (0,\pm 1,..,0)\}$.  \newline
Before proceeding, let us prove the following
\begin{lemma}
\begin{eqnarray}\label{Domjale}
\mathcal{D}((-L_0)^{1/2}) = \{u \in L^2(S^n) : X_{ij}u \in L^2(S^n), X_{ij} \neq X\}.
\end{eqnarray}
\end{lemma}
\begin{proof}
We start by referring to Proposition 1.10, Chapter 8 of ~\cite{T4}, which gives a characterization of $\Cal{D}(A^{1/2})$, where $A$ is a non-negative, unbounded self-adjoint operator on a Hilbert space $H$ constructed by the Friedrichs method.\newline
In the notation of the said proposition, here $H = L^2(S^n)$. Also, let
\[ 
H_1 =  \{u \in L^2(S^n) : X_{ij}u \in L^2(S^n), X_{ij} \neq X\}.\]
\[(u, v)_{H_1}  = (u, v) + \sum_{X_{ij} \neq X} (X_{ij}u, X_{ij}v). \]
$J$ is the natural inclusion $H_1 \to L^2(S^n)$. Then we have that
\begin{align*}
\Cal{D}(-L_0) = \{ u \in H_1: v \mapsto (u, v) + \sum_{X_{ij} \neq X} (X_{ij}u, X_{ij}v) \text{  is continuous in   } v\\
 \forall v \in H_1 \text{   in the  } L^2\text{-norm} \}.
\end{align*}
Now, if we can prove that $H_1$ is a Hilbert space with inner product $(.,)_{H_1}$, then the conclusion of Proposition 1.10 ( Chapter 8 of ~\cite{T4}) gives that $\mathcal{D}((-L_0)^{1/2}) = H_1$.\newline
Now, call $H_{ij} = \Cal{D}(X_{ij}) = \{u \in L^2 : X_{ij}u \in L^2\}$, which becomes a Hilbert space with graph inner product $(u, v)_{ij} = (u, v) + (X_{ij}u, X_{ij}v)$. Then, 
\[
H_1 = \bigcap_{i < j}H_{ij}\]
will become a Hilbert space with the norm $(.,.)_{H_1}$. This is because, given a Cauchy sequence in $H_1$, it becomes a Cauchy sequence in each $H_{ij}$, and since the above intersection is finite, we can select a subsequence which is convergent in every $H_{ij}$. Also, the limit of this subsequence must be the same in every $H_{ij}$, because of the shared component $(.,.)$ ($L^2$ inner product) in each $(.,.)_{ij}$. The limit then lies in the intersection $H_1$, which proves that $H_1$ is a Hilbert space with inner product $(.,.)_{H_1}$.
\end{proof}
With that in place, we take a function $u \in \mathcal{D}((-L_0)^{1/2}))$ having small support in a neighborhood around any of above poles, say, without loss of generality, $(1, 0,...,0)$, and project it down to $\mathbb{R}^n$. This produces a compactly supported projected function on $\RR^n$, still called $u$ with mild abuse of notation, such that 
\begin{eqnarray}\label{p}
u \in H^{1/2}(\mathbb{R}^n), \mbox{    } \partial_{x_i} u \in L^2(\mathbb{R}^n), \text{    }\forall i \in \{2, 3,..., n\}.
\end{eqnarray}
Now, observe that (\ref{p}) implies, after Fourier transforming, $(\xi_1^2 + \xi_2^2 + ..... + \xi^2_n)^{1/4}\hat{u} \in L^2(\mathbb{R}^n)$ and $\xi_i\hat{u} \in L^2(\mathbb{R}^n) \mbox{   }$ for all $i \in \{2, 3,..., n\}$. That means
\[
\hat{f} = (\xi_1^2 + \xi_2^4 + .... + \xi_n^4)^{1/4}\hat{u} \in L^2(\mathbb{R}^n).
\]
We label $u = k\ast f$, where 
\[
\hat{k} = (\xi_1^2 + \xi_2^4 + .... + \xi_n^4)^{-1/4}.
\]
This means that \label{sym7}
\[
k \in C^\infty(\mathbb{R}^n \setminus \{0\}).
\]
Let us quickly justify this: choose $\psi (\xi)$, a bump function that is identically equal to $1$ around the origin. Then $\psi \hat{k}$ is compactly supported, which means $\widehat{(\psi \hat{k})}$ is smooth. So it suffices to prove that $\widehat{((1 - \psi)\hat{k})} \in C^\infty (\RR^n \setminus \{0\})$. \newline
Calling $g = (1 - \psi)\hat{k}$, we see that $g$ is smooth. Now, choose $x \neq 0$, and let $\varphi$ be a smooth bump function around $x$ that is zero at the origin. We need to prove that $\varphi \hat{g}$ is smooth, or equivalently, $\hat{\varphi}\ast g$ vanishes faster than powers of $|\xi|$. \newline
Define $\eta (x) = |x|^{-2m}\varphi(x)$. We see that $\eta (x) \in C^\infty_c(\RR^n)$. Then we have, $\hat{\varphi} = (-\Delta)^m \hat{\eta}$, and $\hat{\varphi}\ast g = ((-\Delta)^m\hat{\eta})\ast g = \hat{\eta}\ast ((-\Delta)^m g)$. Now, no matter how high $m$ is, $\hat{\eta}$ is always Schwartz. Since we can choose $m$ as large as we want, we can make $(-\Delta)^m g$ decay as fast as we want, which gives that $\hat{\varphi}\ast g$ decays faster than powers of $|\xi|$ as infinity.
\newline
Also, $k$ satisfies the anisotropic homogeneity 
\beq\label{Aniho}
k(\delta^2x_1, \delta x_2,..., \delta x_n) = \delta^{-n}k(x_1, x_2,..., x_n).
\eeq
Define $\Omega_0 = \{(x_1,...., x_n): 1/2 \leq |x|^2 < 1\}$ and define $\Omega_j$ for $j \in \mathbb{Z}$ as the image of $\Omega_{j - 1}$ under the map
\[
(x_1, x_2,...., x_n) \mapsto (2^{-1}x_1, 2^{-1/2}x_2,...., 2^{-1/2}x_n).
\]
Using (\ref{Aniho}), we have 
\beq\label{newc}
|k| \leq C2^{nj/2} \mbox{    on    } \Omega_j,
\eeq
where \label{sym30}
\[
|\Omega_j| = 2^{-(n + 1)/2}|\Omega_{j - 1}| = C2^{-((n + 1)/2)j}.
\]
Set 
\[
k_1 = k \text{  on   } \bigcup_{j \geq 0} \Omega_j, \text{  } 0 \text{   elsewhere}
\]
and 
\[
k_2 = k \text{  on   } \bigcup_{j < 0} \Omega_j, \text{  } 0 \text{   elsewhere},
\]
so that $k = k_1 + k_2$. Also, let $u_l = k_l * f, l = 1, 2$.\newline
By (\ref{newc}), we have
\begin{eqnarray}
\int|k_1|^rd\RR^n \leq C\sum_{j \geq 0} 2^{njr/2 - ((n + 1)/2)j} < \infty
\end{eqnarray}
when $r < \frac{n + 1}{n}$. 
Also, 
\begin{eqnarray}
\int|k_2|^rd\RR^n \leq C\sum_{j < 0} 2^{njr/2 - ((n + 1)/2)j} < \infty
\end{eqnarray}
when $r > \frac{n + 1}{n}$.
Now by using Young's inequality for convolutions, we have,
$u_1 \in L^q$, where $q \in [2, \frac{2(n + 1)}{(n - 1)})$ and $u_2 \in L^q$, where $q \in (\frac{2(n + 1)}{(n - 1)}, \infty)$. But $u_2 = u - u_1 \Rightarrow u_2 \in L^2$, and by interpolation, $u_2 \in L^q, q \in [2, \infty)$. So, finally,  $u \in L^q$, where $q \in [2, \frac{2(n + 1)}{(n - 1)})$. So, in our previously introduced notation, $q_* = \frac{2(n + 1)}{(n - 1)}$.\newline

\subsection{The endpoint case $q = \frac{2(n + 1)}{(n - 1)} = 6$ for $n = 2$.}
In the special case of $n = 2$, our setting is now the sphere $S^2$. We already have
\begin{eqnarray}
\mathcal{D}((-L_0)^{1/2}) \subset L^q(S^2), \mbox{     } \forall q \in [2, 6).
\end{eqnarray}
Here we extend the above inclusion up to $q = 6$ and also argue that this is sharp. We have
\begin{lemma}\label{Pro-op}({\bf Optimal embedding and sharpness})\label{seshlabel}
\[\mathcal{D}((-L_0)^{1/2}) \subset L^6(S^2).\] Also, this embedding is sharp. That is,
\[\mathcal{D}((-L_0)^{1/2}) \subset L^q(S^2) \Longrightarrow q \leq 6.\]
\end{lemma}
\begin{proof}
We start by observing that, similar to (\ref{Domjale}) above,
\[
\mathcal{D}((-L_0)^{1/2}) = \{u \in L^2(S^2): Yu, Zu \in L^2(S^2)\},
\]
where $Y, Z$ are respectively the restrictions on $S^2$ of the vector fields that generate rotations about the $y$-axis and the $z$-axis in $\RR^3$.\newline 
Ellipticity of $-L_0$ away from the poles $(0, \pm 1, 0), (0, 0, \pm 1)$ implies
\[
u \in \mathcal{D}((-L_0)^{1/2}) \Rightarrow \varphi u \in H^1(S^2),
\]
where $S = S^2 \setminus \{(0, \pm 1, 0), (0, 0, \pm 1)\}$ and $\varphi \in C^\infty_c(S)$. \newline
With that in place, we take a function $u \in \mathcal{D}((-L_0)^{1/2}))$ having small support in a neighborhood around any of the points in $S$, say, without loss of generality, $(0, 1, 0)$ and project it to $\mathbb{R}^2$ in the following way: let $\gamma_{xy}, \gamma_{yz}$ and $\gamma_{zx}$ denote the great circles on $S^2$ lying on the $xy, yz$ and $zx$-planes respectively. Then the projection takes a neighborhood of $(0, 1, 0)$ in $\gamma_{xy}$ onto the $y$-axis and a neighborhood of $(0, 1, 0)$ in $\gamma_{yz} $ onto the $x$-axis. This produces a compactly supported projected function on $\RR^2$, still called $u$ with mild abuse of notation, such that 
\beq\label{onep1}
u \in H^{1/2}(\mathbb{R}^2), \partial_yu \in L^2(\RR^2).
\eeq 
Also, since we have already asserted that $Zu \in L^2(S^2)$, this will give 
\beq \label{onep2}
(x\partial_y - y\partial_x)u \in L^2(\RR^2).
\eeq Since $u$ is compactly supported, $\partial_yu \in L^2(\RR^2) \implies x\partial_yu \in L^2(\RR^2)$, which, coupled with the last fact, implies, $y\partial_xu \in L^2(\RR^2)$. \newline
We will use the first two pieces of data, namely, $u \in H^{1/2}(\mathbb{R}^2)$ and $\partial_yu \in L^2(\RR^2)$. We observe that this means 
\beq\label{MRen}
u \in H^{1/2}_x(L^2_y) \cap L^2_x(H^1_y).
\eeq
Let us first justify this. \newline
$u \in L^2_x(H^1_y)$ actually means $\V\V u\V_{H^1_y}\V_{L^2_x} < \infty \ \Leftrightarrow \V\V (1 + \eta^2)^{\frac{1}{2}}\hat{u}^{y}\V_{L^2_\eta}|\V_{L^2_x} < \infty$, where $\hat{u}^{y}$ represents Fourier transform with respect to $y$, that is, $\hat{u}^{y}$ is now a function of $x$ and $\eta$.\newline
Now,
\begin{align*}
\V \V u\V_{H^1_y}\V_{L^2_x} & = \V \V(1 + \eta^2)^{\frac{1}{2}}\hat{u}^y\V_{L^2_x}\V_{L^2_\eta} \\
& = \V (1 + \eta^2)^{\frac{1}{2}}\V\hat{u}^{y}\V_{L^2_x}\V_{L^2_\eta} = \V (1 + \eta^2)^{\frac{1}{2}}\V\hat{u}\V_{L^2_\xi}\V_{L^2_\eta}\\
& = \V\V (1 + \eta^2)^{\frac{1}{2}}\hat{u}\V_{L^2_\xi}\V_{L^2_\eta} < \infty,
\end{align*}
since $(1 + \eta^2)^{1/2}\hat{u} \in L^2(\RR^2)$.\newline
Similarly, $u \in H^{1/2}_x(L^2_y) \Leftrightarrow \V (1 + \xi^2)^{1/4}\hat{u}^x(\xi)\V_{L^2_y} \in L^2_\xi$, where $\hat{u}^x$ means Fourier transform with respect to $x$ only. This holds iff $\V  \V ((1 + \xi^2)^{1/4}\hat{u}^x(\xi)\V_{L^2_y}\V_{L^2_\xi} < \infty$, which follows from $u \in H^{1/2}(\RR^2)$. This implies (\ref{MRen}).\newline
Now we propose to use interpolation (\cite{LM}, Chapter 4 has a detailed treatment of these sorts of spaces and allied results). By interpolation, we can say that for $\theta \in [0, 1]$
\beq\label{Part}
u \in H^{\frac{1}{2}\theta}(H^{1 - \theta}_y),
\eeq
where $H^r_x(H^s_y)$ denotes $H^s_y$-valued $H^r$-functions of $x$. This is because, 
\begin{align*}
u \in H^{\frac{1}{2}\theta}(H^{1 - \theta}_y) & \Leftrightarrow (1 + \xi^2)^{\theta/4}\hat{u}(\xi, y) \in L^2_\xi (H^{1 - \theta}_y)\\
& \Leftrightarrow (1 + \xi^2)^{\theta/4}(1 + \eta^2)^{(1 - \theta)/2}\hat{u}(\xi, \eta) \in L^2_\xi (L^2_\eta).
\end{align*}
But this follows by interpolation from $(1 + \xi^2)^{1/4}\hat{u}(\xi, \eta) \in L^2_\xi(L^2_\eta)$ and $(1 + \eta^2)^{1/2}\hat{u}(\xi, \eta) \in L^2_\xi(L^2_\eta)$.
Now, from (\ref{Part}), particularly for $\theta = 2/3$, we have
\[
u \in H^{1/3}_x(H^{1/3}_y).
\]
Now, when we use Sobolev embedding in one dimension, we know that  $H^{1/3}$ embeds in $L^6$.  That means, $u \in L^6_x(L^6_y)$, which implies, $u \in L^6(\mathbb{R}^2)$.
\newline
We will now prove the next part of the lemma: that the estimate of $u \in L^6$ as obtained above is sharp. Since $u$ has compact support, $\pa_y u \in L^2(\RR^2) \Rightarrow x\partial_y u \in L^2(\mathbb{R}^2)$, which, coupled with (\ref{onep2}) implies $y\partial_xu \in L^2(\mathbb{R}^2)$. We also have (\ref{onep1}). \newline
Let us define
\[
u(r, \sigma, a, b, x, y) =  r^\sigma u(r^ax, r^by).
\]
In the ensuing calculations, we will write, when convenient, $u(r, \sigma, a, b)$ for $u(r, \sigma, a, b, x, y)$ for ease of handling symbols.\newline
We have, 
\begin{align*}
\V \partial_yu(r, \sigma, a, b)\V^2_{L^2} & = \int_{\mathbb{R}^2}|\partial_yu(r, \sigma, a, b, x, y)|^2 dxdy = \int_{\mathbb{R}^2}|\partial_yr^\sigma u(r^ax, r^by)|^2dxdy\\ & = \int_{\mathbb{R}^2}|r^b\partial_{z_2}r^\sigma u(z_1, z_2)|^2r^{-a}r^{-b}dz_1dz_2\\ & = r^{b + 2\sigma - a}\V \partial_yu\V^2_{L^2}.
\end{align*}
Similarly, we can calculate,
\[
\V y\partial_xu(r, \sigma, a, b)\V^2_{L^2} = r^{2\sigma + a - 3b}\V y\partial_xu\V^2_{L^2}.
\]
Now, 
\begin{align*}
\hat{u}(r, \sigma, a, b, \xi, \eta) & \simeq \int_{\mathbb{R}^2}u(r, \sigma, a, b, x, y)e^{-i(\xi x + \eta y)}dxdy\\
& = \int_{\mathbb{R}^2}r^\sigma u(r^ax, r^by)e^{-i(\xi x + \eta y)}dxdy\\
& = \int_{\mathbb{R}^2}r^\sigma u(z_1, z_2) e^{-i(\frac{\xi}{r^a}z_1 + \frac{\eta}{r^b}z_2)}r^{-a}r^{-b}dz_1dz_2 \\
& = r^{\sigma - a - b}\hat{u}(r^{-a}\xi, r^{-b}\eta).
\end{align*}
So, 
\begin{align*} 
\V u(r, \sigma, a, b)\V^2_{H^{1/2}} & = \int_{\mathbb{R}^2}(1 + \xi^2 + \eta^2)^{1/2} r^{2\sigma - 2a - 2b}|\hat{u}(r^{-a}\xi, r^{-b}\eta)|^2d\xi d\eta\\
& = r^{2\sigma - 2a - 2b}\int_{\mathbb{R}^2}(1 + r^{2a}\theta^2 + r^{2b}\phi^2)^{1/2}|\hat{u}(\theta, \phi)|^2r^ar^b d\theta d\phi, \theta = r^{-a}\xi, \phi = r^{-b}\eta\\
& = \int_{\mathbb{R}^2}(r^{2(2\sigma - a - b)} + r^{4\sigma - 2b}\xi^2 + r^{4\sigma - 2a}\eta^2)^{1/2}|\hat{u}(\xi, \eta)|^2d\xi d\eta.
\end{align*}
We will want to compare this estimate with $\V u\V^2_{H^{1/2}} = \int_{\mathbb{R}^2}(1 + \xi^2 + \eta^2)^{1/2}|\hat{u}(\xi, \eta)|^2d\xi d\eta$.\newline
Also, on calculation, $\V u(r, \sigma, a, b)\V^p_{L^p} = r^{\sigma p - a - b}\V u\V^p_{L^p}$. \newline
Now, suppose that 6 is not a sharp exponent. We begin by choosing a $u \in \Cal{D}((-L_0)^{1/2})$ satisfying $u \in L^{6 + \varepsilon}$, where $\varepsilon > 0$. In the above equations, we let $\sigma = 1$ by observation. Then we see that for $a = 4$ and $b = 2$ (and calling $u(r, 1, 4, 2) = u_r$), we see that $\V \partial_y u_r\V_{L^2} = \V \partial_y u\V_{L^2}, \V y\partial_x u_r\V_{L^2} = \V y\partial_x u\V_{L^2}$ and $\V u_r\V_{H^{1/2}} \leq \V u\V_{H^{1/2}}$ when $r \geq 1$. \newline
On calculation, $\V u_r\V^{6 + \varepsilon}_{L^{6 + \varepsilon}} = r^{\varepsilon}\V u\V^{6 + \varepsilon}_{L^{6 + \varepsilon}}$. Clearly, as we let $r$ increase, the left hand side increases, with a fast decreasing support, since the support of $u$ was compact to begin with. \newline
Finally, to get a contradiction, we just have to take a sequence of $u_r$ for fast increasing $r$, with disjoint supports, and sum them up. To be precise, we already have $\V u_r\V_{L^{6 + \varepsilon}} = Kr^{\theta}$, where $K = \V u\V_{L^{6 + \varepsilon}}$ is a constant and $\theta = \frac{\varepsilon}{6 + \varepsilon} > 0$.\newline
Define a new function $u^*$ by $u^* = \Sigma \frac{1}{2^n} v_{r_n}$, where $r_n$ is chosen such that $2^{n - 1} \leq r_n^\theta < 2^n$ and $v_{r_n}$ is obtained by a translate of $u_{r_n}$ parallel to the $x$-axis, in such a way that all the $v_{r_n}$ have disjoint support. That way, we still preserve control over $\V\partial_y u^*\V_{L^2}, \V y\partial_x u^*\V_{L^2}$ and $\V u^*\V_{H^{1/2}}$, but the $L^{6 + \varepsilon}$-norm of $u^*$ blows up, contrary to our assumption.
\end{proof}
\subsection{Higher regularity in case of nonsmooth nonlinearity on $S^2$: a particular calculation}
It has already been shown that \[u \in \mathcal{D}((-L_0)^{1/2}) \Rightarrow u \in L^6.\] Now if $u$ solves (\ref{NLKGw}), then we can do better. A specific case ($p = 3$) has been worked out in ~\cite{T1} and it has been shown (using an elliptic bootstrapping argument) that $u$ is then smooth. Now, if $p$ is not an odd integer, we cannot expect a similar smoothness, because the nonlinearity of (\ref{NLKGw}) itself is then not smooth. However, we can expect higher Sobolev spaces and, in turn, higher $L^r$ spaces (by Sobolev embedding) for $u$ when $p$ is not an odd integer. Here, we calculate one explicit case, namely, $p = 4$. A word is in order regarding this choice. Firstly, let $P = -L_{2\lambda} + (m^2 - \lambda^2)$ and $F(u) = K|u|^{p - 1}u$ whence $Pu = F(u)$. Let $3 < p < 6$. \newline
Then, $F(u) \in L^{6/p}$. On calculation, $(L^{6/p})^{*} = L^{\frac{6}{6 - p}}$. By using the Sobolev embedding theorem, we can find a $\delta > 0$ such that $H^\delta \subset L^{\frac{6}{6 - p}}$. On calculation, this happens when $\delta > p/3 - 1$. So, by duality,
\beq\label{ebark1}
F(u) \in L^{6/p} \subseteq \displaystyle\bigcap_{\delta > p/3 - 1}H^{- \delta},
\eeq
which means
\beq\label{ebark2}
u = P^{-1}(F(u)) \subseteq \displaystyle\bigcap_{\delta > p/3 - 1}H^{1 - \delta}.
\eeq
The above claim comes from the fact that $P$ is hypoelliptic from H$\ddot{\mbox{o}}$rmander's condition. We see that $P$ is hypoelliptic if $L_{2\lambda}$ is. Since $-\Delta + X^2 = Y^2 + Z^2$, and $[Y, Z] = X$ (see the more general demonstration on page \pageref{Hypo}), $L_{2\lambda}$ is hypoelliptic. Also, from Theorem 1.8, Chapter XV of ~\cite{T2}, P is hypoelliptic  with the loss of a single derivative.\newline
Note that we already know that $u \in H^{1/2}$. So this bootstrapping process yields something better than what we started with only when $\delta < 1/2$, or equivalently, $3 < p < 9/2$. So for an explicit demonstration we have chosen $p = 4$.\newline
When $p = 4$, according to previous calculation, 
\begin{eqnarray}\label{ebark3}
u = P^{-1}(F(u)) \subseteq \displaystyle\bigcap_{\delta > 1/3}H^{1 - \delta} = \displaystyle\bigcap_{\varepsilon > 0}H^{2/3 - \varepsilon}.
\end{eqnarray}
As argued before, $u \in H^1$ when $u$ is supported away from the poles. So, choose neighborhoods around the ``north pole'' of the 2-sphere in the following manner: $U, V$ and $W$ are open neighborhoods such that $V \subset \overline{V} \subset W \subset \overline{W} \subset U$. Also choose a smooth bump function $\phi$ such that $\text{supp }\phi \subset \overline{W}$ and $\phi \equiv 1$ on $\overline{V}$. Note that $\phi u$ satisfies (\ref{NLKGw}) inside $V$, so with a suitably chosen $\phi$ we can ensure that $\phi u \in \displaystyle\bigcap_{\varepsilon > 0}H^{2/3 - \varepsilon}(U)$.  \newline
Now we are going to determine if $\phi u$ belongs in a higher Sobolev space. Surely, $\phi u$ will not solve (\ref{NLKGw}) on $U$, but that is fine. All we want to investigate is the behavior of $u$ around the pole, which can be tracked by the behavior of $\phi u$ inside $V$. Now, projecting down $U$ on the plane, we see that the projection of $\phi u$, called $v$, satisfies $\partial_y v \in L^2$ and $v$ has compact support. This implies, by the interpolation procedure on mixed Sobolev spaces carried out in the proof of Lemma \ref{seshlabel}, that,
\beq\label{r10}
v \in L^r, r < 10.
\eeq
Let us argue how this goes. We know $v \in \bigcap_{\varepsilon> 0} H^{2/3 - \varepsilon}$ and $ \pa_yv \in L^2$. By arguments outlined in the proof of Lemma \ref{seshlabel}, that means
\beq
v \in H^{2/3 - \varepsilon}_x(L^2_y) \cap L^2_x(H^1_y), \forall \varepsilon > 0.
\eeq
By interpolation,
\beq 
v \in H^{(2/3 - \varepsilon)\theta}_x(H^{1 - \theta}_y), \forall \varepsilon > 0, \theta \in [0, 1].
\eeq
Choosing $\theta = \frac{1}{5/3 - \varepsilon}$, we finally get that 
\beq 
v \in H^{(2/3 - \varepsilon)(\frac{1}{5/3 - \varepsilon})}_x(H^{(2/3 - \varepsilon)(\frac{1}{5/3 - \varepsilon})}_y), \forall \varepsilon > 0.
\eeq
Sobolev embedding then gives (\ref{r10}).\newline
(\ref{r10}), in turn, through the bootstrapping procedure given by (\ref{ebark1}), (\ref{ebark2}) and (\ref{ebark3}), implies that $v \in \displaystyle\bigcap_{\varepsilon > 0}H^{4/5 - \varepsilon}$, which means, $u \in \displaystyle\bigcap_{\varepsilon > 0}H^{4/5 - \varepsilon}$. This is a gain in regularity.
\section{\bf $\langle X, X\rangle = 1$ vis-a-vis contact structures}\label{CP}
When $\lambda = 0$, we can extend the analysis done till now to a much larger class of manifolds to get an existence result for the NLKG. These are the class of so-called K-contact manifolds with an associated metric $g$. To recall the definitions:
\begin{definition}
A contact manifold $(M^{2n + 1}, \eta)$, with characteristic or Reeb vector field $\xi$ has an associated Riemannian metric $g$ if $\eta(X) = g(X, \xi)$ and there exists a tensor field of type $(1, 1)$ such that $\phi^2 = - I + \eta \otimes \xi$, $d\eta(X, Y) = g(X, \phi Y)$. This is called a contact metric structure. 
A contact metric structure $(M, \phi, \xi, \eta, g)$ is called a K-contact structure if the Reeb vector field $\xi$ is a Killing field (with respect to $g$). 
\end{definition}
Examples of K-contact manifolds abound in literature (see ~\cite{B}). This includes in particular the Sasakian manifolds, and more particularly, the odd dimensional unit spheres (the word ``unit'' is important, for otherwise the metric from Euclidean space will not be an associated one). \newline
Under this setting, let us consider the Reeb vector field, say $X$, on $(M, \phi, X, \eta, g)$ and consider $L_0 = \Delta - X^2$, where $\Delta$ is the Laplace-Beltrami operator on $M$. Since $\eta(X) = 1$, we see that $X$ has unit length throughout. It has been shown by ~\cite{BD}, using a result of Radkevic, that $L_0$ is subelliptic of order $1/2$. So by a result of Kohn-Nirenberg, we can say that the subLaplacian $L_0$ is hypoelliptic with the loss of a single derivative 
(also from the subelliptic estimates it can be proved that $L_0$ has discrete spectrum, as has been proved in ~\cite{BD}).\newline
Now, we can either make the technical assumption that $L_\alpha = L_0 - i\alpha X$ is negative semidefinite, 
or we can treat travelling wave solutions of NLKG with $\lambda = 0$. Under this assumption, using arguments quite similar to what has gone before, we can then give an existence statement of the NLKG on K-contact manifolds with an associated metric. \newline
However, as we show below, in specific cases of manifolds with locally contact structures where the spectrum can be evaluated explicitly, we can say more. We again make a mention of ~\cite{T1}, which investigates this phenomenon on $M = S^3$. The setup there is as follows: the group $SU(2)$, with its bi-invariant Riemannian metric, is isometric to $S^3$ and is a cover of $SO(3)$. Call $X, Y$ and $Z$ the left-invariant vector fields on $SU(2)$, covering vector fields on $SO(3)$ that generate $2\pi$-periodic rotations of $R^3$ about the $x, y$ and $z$-axis, respectively. Then the Laplacian on $S^3$ can be written as 
\begin{eqnarray}
\Delta = X^2 + Y^2 + Z^2,
\end{eqnarray} 
\begin{eqnarray}
\langle X, X\rangle = \langle Y, Y\rangle = \langle Z, Z\rangle = 1 \mbox{  on  } S^3,
\end{eqnarray}
so a solution to the NLKG under these conditions can be called sonic wave solutions.\newline
Call, as before, $L_0 = \Delta - X^2$. The main result is as follows:
\begin{proposition}(~\cite{T1})
\beq\label{erseshkothay}
\Cal{D}((-L_0)^{1/2}) \hookrightarrow L^q(S^3), \text{   } q \in [2, 4),
\eeq
and the inclusion is compact. Then, under the assumptions 
\beq\label{sharpt2}
1 < p < 3
\eeq and
\[
|\lambda| < \frac{1}{2}, m^2 > \lambda^2,\]
(\ref{NLKGw}) has a nonzero solution $u \in \Cal{D}((-L_0)^{1/2})$.
\end{proposition}
\subsection{Estimates on $S^7$}
Here we will attempt to give an analysis on 
$M = S^7$ along somewhat similar lines. Note, however, that the above analysis on $S^3$ depends heavily on the fact that it can be given a Lie group structure. As is well known, $S^1$ and $S^3$ are the only such spheres. So, in certain details, our investigation has to diverge. We also mention that here we will give a sharp estimate corresponding to (\ref{erseshkothay}), and our method will also prove that the estimate (\ref{erseshkothay}) in ~\cite{T1} is actually sharp.\newline In our situation, what will come in handy is the fact that $S^7$ has certain homogeneous space-like properties\footnote{The reason being, $S^7$ has a similar connection with the octonions as has $S^3$ with the quaternions.}, one of them being: $S^7$ has a global orthonormal frame of Killing fields which do not commute. This allows us to write the Laplacian on $S^7$ as $\Delta = X^2_1 + X^2_2 + .... + X^2_7$ where $X_i$ generate the global orthonormal Killing frame. If we consider $S^7$ as the sphere in $\RR^8$ centred at the origin, then at the point $P \in S^7$ having coordinates $P = (x_1, x_2,.., x_8)$, we can take \[
X_1 = (x_2, -x_1, x_4, -x_3, x_6, -x_5, x_8, -x_7),\] \[X_2 = (x_3, -x_4, -x_1, x_2, x_7, -x_8, -x_5, x_6),\] \[ X_3 = (x_4, x_3, - x_2, -x_1, -x_8, -x_7, x_6, x_5),\]\[ X_4 = (x_5, -x_6, -x_7, x_8, -x_1, x_2, x_3, -x_4),\]\[ X_5 = (x_6, x_5, x_8, x_7, -x_2, -x_1, -x_4, -x_3),\]\[X_6 = (x_7, -x_8, x_5, -x_6, - x_3, x_4, -x_1, x_2),\]\[X_7 = (x_8, x_7, -x_6, -x_5, x_4, x_3, -x_2, -x_1).\]
Let us mention as a side remark that the existence of global Killing frames is a stronger condition than being able to write the Laplacian as a sum of squares in local coordinates; for details, see ~\cite{DN}. They show, among other things, that $S^1, S^3$ and $S^7$ are the only spheres that possess global orthonormal Killing frames. \newline
Letting $X = X_1$ without loss of generality, $L_0 = \Delta - X^2$, and $L_\alpha = L_0 - i\alpha X$, and following the lines of calculation in Section \ref{SPS}, 
we see that $-L_\alpha$ is positive semidefinite when $|\alpha| < 6$, whence by the spectral theorem we can define $(-L_0)^{1/2}$.\newline
Now, let $q_*$ be the optimal (greatest) number such that
\beq\label{OPTI3}
\mathcal{D}((-L_0)^{1/2}) \subset L^q(S^7), \forall \mbox{   } q \in [2, q_*)
\eeq
is compact. We can see that if the inclusion (\ref{OPTI3}) holds, it is continuous via the closed graph theorem applied to the inclusion operator. For the norm on $\Cal{D}((-L_{0})^{1/2})$, we use the graph norm given by
\[
\V u\V_{\Cal{D}((-L_{0})^{1/2})}^2 = ((-L_0)^{1/2} u, (-L_0)^{1/2} u) + (u, u),\]
which turns $\Cal{D}((-L_{0})^{1/2})$ into a Hilbert space (see Proposition 1.4 of ~\cite{Sc}). Then, as argued before (see page \pageref{compactn}), proving the inclusion (\ref{OPTI3}) will prove the compactness of the same.\newline
So, our existence result is:
\begin{proposition}\label{esseven} ({\bf Existence result on $S^7$})
With $X = X_1$ as above (without loss of generality), assume
\[
2 < p + 1 < q_*\text{   (as in  } (\ref{OPTI3})).\]
Also assume,
\[
|\lambda| < 3,\text{    }m^2 > \lambda^2.
\]
Then, given $K > 0$, the equation
\[
-L_{2\lambda}u + (m^2 - \lambda^2)u = K|u|^{p - 1}u
\]
has a nonzero solution $u \in \mathcal{D}((-L_0)^{1/2})$.
\end{proposition}
\begin{proof}
We apply H{\"o}rmander's sum of squares (see \cite{H}) on $L_0 = \Delta - X^2$ where $X = X_1$ without loss of generality, to conclude that it is hypoelliptic with the loss of one derivative.\newline 
Then we use the fact that 
\begin{eqnarray}\label{GAPRE1}
F_{m, \lambda, X}(u) = (-L_{2\lambda} u, u) + (m^2 - \lambda^2)(u, u) \cong \V u\V^2_{\mathcal{D}((-L_{2\lambda})^{1/2})},
\end{eqnarray}
where $\V u\V^2_{\mathcal{D}((-L_{2\lambda})^{1/2})}$ is the graph norm given by 
\[
\V u\V^2_{\mathcal{D}((-L_{2\lambda})^{1/2})} = ((-L_{2\lambda})^{1/2} u, (-L_{2\lambda})^{1/2} u) + (u, u),
\]
which turns $\mathcal{D}((-L_{2\lambda})^{1/2})$ into a Hilbert space. \newline
Let
\[
I_\beta = \inf \{F_{m, \lambda, X}(u) : u \in \mathcal{D}((-L_{2\lambda})^{1/2})\},
\]
under the constraint (\ref{ic}). Now, take a sequence of functions $u_\nu \in \mathcal{D}((-L_{2\lambda})^{1/2})$ such that $F_{m, \lambda, X}(u_\nu) \to I_\beta$. Then, (\ref{GAPRE}) implies that $\V u_\nu\V_{\Cal{D}((-L_{2\lambda})^{1/2})}$ is uniformly bounded, which in turn means (a subsequence of) $u_\nu$ weakly converges to $u \in \mathcal{D}((-L_{2\lambda})^{1/2})$. By virtue of the compactness of (\ref{OPTI2}), $u_\nu$ has a subsequence, still called $u_\nu$ with mild abuse of notation, that is strongly $L^{p + 1}$ convergent to $u$, where $p + 1 < q_*$, meaning that $\V u\V^{p + 1}_{L^{p + 1}} = \beta$, as in (\ref{ic}). Also, 
\[
F_{m, \lambda, X}(u) \leq \liminf F_{m, \lambda, X}(u_\nu).
\]
So, $u \in \mathcal{D}((-L_{2\lambda})^{1/2}))$ gives a constrained minimizer to $F_{m, \lambda, X}(u)$ subject to (\ref{ic}).
The constrained minimizer will give a solution to (\ref{NLKGw}), as wanted.
\end{proof}

\begin{remark}
Note that the above proposition is {\bf not} a special case of Proposition (\ref{Pro}), because of the very different natures of the sub-Laplacians and merits its own different proof. 
\end{remark}

\subsection{Optimal exponent $q^*$}\label{sevenop}
For us, the problem is again to determine how high $q_*$ can be.
In this case the Sobolev Embedding yields, 
\[
H^{1/2}(S^7) \subseteq L^{7/3}(S^7)
\]
However, we have the following
\begin{lemma}\label{sevenop1}
\beq\label{pcwp}
H^{1/2}(S^7) \hookrightarrow L^q, \text{   } q < 8/3,
\eeq
and the inclusion is compact. Also, this is optimal. That is, 
\[
H^{1/2}(S^7) \subset L^{q} \Longrightarrow q < 8/3.
\]
\end{lemma}
\begin{proof}
By similar arguments as outlined just before Proposition \ref{Pro}, we can assert that proving the inclusion (\ref{pcwp}) will automatically show that it is compact.\newline
We observe that we can localize an open set of $S^7$ to an open set in $\mathbb{H}^7$ with coordinates $(p, q, t)$ such that 
\[
(1 - L_0)^{-1/2}u(z) = k(z, .)\ast u(z)
\]
plus lower order terms, where $\hat{k} \in C^{\infty}(\mathbb{R}^7 \setminus \{0\})$ satisfies the following anisotropic homegeneity:
\[
\hat{k}(sx, sy, s^2\tau) = s^{-1}\hat{k}(x, y, \tau),
\]
which means, 
\begin{eqnarray}\label{S}
k(\lambda p, \lambda q, \lambda^2 t) = \lambda^{-7}k(p, q, t),
\end{eqnarray}
where $p$ and $q$ refer to triplets of real numbers and $(p, q, t) \in \mathbb{H}^7$, the Heisenberg group of dimension 7.
For details on these kinds of calculations, see ~\cite{T3} and ~\cite{FS}. 
See also ~\cite{P}.\newline
We now want to see what (\ref{S}) implies.  Let us assume without loss of generality that 
\[
k(p, q, t) = 1 \mbox{   when   } |p|^2 + |q|^2 + t^2 = 1.
\]
If we let $\Omega_0 = \{(p, q, t): 1 \leq k(p, q, t) \leq 8\}$ and define $\Omega_k$ inductively as the image of $\Omega_{k - 1}$ under the map
\[
(p, q, t) \mapsto (2^{-1}p, 2^{-1}q, 2^{-2}t),
\]
we see that 
\[
\mbox{Vol   }\Omega_k = 2^{-8}\mbox{Vol   }\Omega_{k - 1} = C2^{-8k}.
\]
Hence, with $B = \cup_{k \geq 0}\Omega_k$, on calculation, we have
\[
\int_B |k(p, q, t)|^p dV \leq C\sum_{k \geq 0} 2^{7pk}2^{-8k} < \infty,
\]
which means that 
\[
k \in L^p_{loc}(\mathbb{H}^7) \mbox{    for    } p < \frac{8}{7}.
\]
By interpolation we find that $k \ast u \in L^q_{loc}(\mathbb{H}^7)$ for $q < \frac{8}{3}$. \newline
We will prove that the above estimate of $p < \frac{8}{7}$ is optimal. The derivation will employ a similar scaling trick as used before, though the calculations will be simpler. Observe that the integration here takes place on a Heisenberg group, but we note that the ordinary Lebesgue measure on $\mathbb{R}^{n}$ gives the Haar measure on $\mathbb{H}^n$, and the Heisenberg group is unimodular. So we can transform coordinates and repeat similar calculations that we did for $\mathbb{R}^2$.\newline
Define 
\beq\label{sharpt1}
k(r, \sigma, p, q, t) = rk(r^\sigma p, r^\sigma q, r^{2\sigma}t).
\eeq
The right hand side is equal to $r.r^{-7\sigma}k(p, q, t) = r^{(1 - 7\sigma)}k(p, q, t)$ by anisotropic homogeneity. Also $k(r, \sigma, p, q, t) \in C^{\infty}(\mathbb{H}^7\setminus 0)$ for all $r, \sigma$. \newline
Now we calculate the local $L^p$-norm of $k(r, \sigma)$ in two different ways:\newline
Firstly,
\begin{align}\label{and1}
\nonumber\V k(r, \sigma)\V^p_{B(0, r)} & = \int_{B(0, r)}|k^r_\sigma|^pdV \\\nonumber
& = \int_{B(0, r)}r^{p - 7\sigma p}|k(p, q, t)|^pdV\\
& = r^{(1 - 7\sigma)p}\V k\V^p_{B(0, r)},
\end{align}
the second step coming from the anisotropic homogeneity.\newline
Again
\begin{align}\label{and2}
\nonumber\V k(r, \sigma)\V^p_{B(0, r)} & = \int_{B(0, r)}|rk(r^\sigma p, r^\sigma q, r^{q\sigma}t)|^pdV\\
\nonumber& = \int_{B(0, r')}r^p|k(p', q', t')|^pr^{-8\sigma}, r' > r\\
& = r^{p - 8\sigma}\V k\V^p_{B(0, r')},
\end{align}
the second step coming from a change of variables. \newline
Comparing (\ref{and1}) and (\ref{and2}), we see that $(1 - 7\sigma)p > p - 8\sigma$, meaning $p < \frac{8}{7}$. This also takes care of the endpoint case of $8/7$.
\end{proof}
\begin{remark}
One observes that similar scaling tricks, as in (\ref{sharpt1}), can be used to show that the estimate of $1 < p < 3$ as in (\ref{sharpt2}) (which is quoted from Proposition 4.2 of ~\cite{T1}) is sharp. 
\end{remark}

\section{\bf Results in the non-compact setting: Main theorems}
\label{NCR}
In this section, we enter into our main theorems, which deal with constrained minimizing solutions of (\ref{NLSw}) and (\ref{NLKGw}) on non-compact manifolds. As mentioned at the outset of this paper, we extend results from ~\cite{T1} and also results proved earlier in this paper to a non-compact setting. More precisely, we will try to repeat the analysis of Section 1 of ~\cite{T1} and Section 1.3 of this thesis in the case of non-compact manifolds $M$. Before we begin, here is a heuristic story. What will cost us most dearly in the non-compact setting is the failure of compact Sobolev embedding. That means, we have to find out means of exercising some control over functions outside a large compact set. We do this in two different ways: one is to consider constrained minimization in a subclass of $H^1(M)$ functions and impose appropriate geometric restrictions on the manifold $M$ that will make elements of the said subclass vanish at infinity, and this will make the failure of the Sobolev embedding a manageable problem. This we do in the case of $F_{m, \lambda, X}$ minimizers in Subsection \ref{MTIS}. Another way to proceed is to use concentration compactness arguments in the presence of a certain geometric homogeneity of the space $M$. In that case, once we prove concentration, we can use the geometric homogeneity of the space to bring all the zones of concentration within a compact region and get compact Sobolev embedding into action. This we implement in the case of constrained $\Cal{E}_{\lambda, X}$ minimizers in Subsection \ref{MTIIS}.
\subsection{$F_{m, \lambda, X}$ minimizers}\label{MTIS}
Here we consider non-compact manifolds $M$ with $C^\infty$ bounded geometry which are of the form $[0, \infty) \times N$. Here $N$ is assumed to be compact and $(n - 1)$ dimensional, and $M$ has the product metric $g = dr^2 + \phi(r) g_N$, where $\phi$ is smooth and positive with $\phi(1) = 1$, and $g_N$ denotes the metric on $N$. 
Also, we assume that $M$ is complete, and all the points $(0, x), x \in N$ are identified to a single point (which could be thought of as the origin). The last two assumptions respectively mean that we do not have to worry about cone points and boundaries.\newline
If $X$ is a Killing field on $(N, g_N)$, consider the push-forward vector field $X_r$ on $(N, \phi(r) g_N)$, that is, $X_r = i_{r_*}X$, $i_r : (N, g_N) \to (N, \phi(r) g_N)$ being the identity map. This induces a Killing field on $M$, still called $X$ by abuse of notation. Note that, for $r \in [0, \infty), x \in N$, $X_{(r, x)} = \sqrt{\phi(r)}X_{(1, x)}$. We will consider only those $M$ which have bounded geometry and those functions $\phi$ such that $\langle X, X\rangle \leq b^2 < 1$. As an example of the kind of space we are talking about, consider the cylinder $[1, \infty) \times S^1$ fitted with a hemispherical cap (diffeomorphic to the closed 2-disc) to make it complete. Here $X$ is given by (slow) rotation about the axis of the cylinder.\newline
In general (see ~\cite{CMMT}, Section 2.3, for example), we should not expect minimizers of $F_{m, \lambda, X}$ on $H^1(M)$ when $M$ is complete and non-compact, even if it has rotational symmetry. However, we can minimize $F_{m, \lambda, X}$ on the class of radial functions which are in $H^1(M)$; that is, we will try to minimize $F_{m, \lambda, X}$ over\label{sym14}
\[
H^1_r(M) = \{u \in H^1(M) : u \mbox{ is a radial function}\}.
\]
A word is in order regarding what is meant by a radial function. Here it means those functions which are dependent only on the variable $r$ running over $[0, \infty)$ of the space $M = [0, \infty) \times N$, i.e., we are considering only those functions $f$ for which $f(r, x) = \varphi(r)$. Also, if $A(r)drdN$ represents the volume form of $M$, then by calculation, we have $A(r) = (\phi(r))^{\frac{n - 1}{2}}$.\newline
To work out constrained $F_{m, \lambda, X}$ minimizers, we first need a lemma:
\begin{lemma}\label{Cond} Consider a non-compact complete manifold $M$ of dimension $n$ satisfying the properties described at the beginning of this subsection. Also, assume a positive lower bound on $\phi(r)$ outside a compact set, say, when $r > 1$. Then, if $f \in H^1_r(M)$, $f$ vanishes at infinity.
\end{lemma}
\begin{proof} 
We start by justifying that $f \in H^1_r(M) \Rightarrow f \in C(M\setminus U)$, where $U$ is a neighborhood of the origin, let us say, without loss of generality, a ball of radius 1. The argument behind this is essentially local. Choose $(r', x') \in M$ and a small open ball $B = (r' - \delta, r' + \delta) \times V$ around $(r', x')$, where $V$ is open in $N$. We can see that $u \in H^1(B) \Rightarrow u(r, x') \in H^1((r' - \delta, r' + \delta))$, and since all the components of the metric tensor $g$ are uniformly bounded on $B$, one-dimensional Sobolev embedding gives $u(r, x') \in C((r' - \delta, r' + \delta)) \Rightarrow u \in C(B)$ . 
Also, since functions in $C(M \setminus U)$ can be uniformly approximated by functions in $C^\infty(M \setminus U)$, we can assert that it is enough to prove the lemma for $f \in C^1(M \setminus U) \cap H^1_r(M)$.\newline
Now, if $f$ does not vanish at infinity, then, there exists an $\varepsilon > 0$ such that no matter what compact set in $M$ we select, $f$ attains a value greater than $\varepsilon$ outside this compact set. By scaling the function if necessary, we can use $\varepsilon = 1$. \newline Let $q_k \in M$ be a sequence of points satisfying the following:\newline
(a) $q_k$ has coordinates $(r_k, x)$, $r_k \in (0, \infty)$, $x \in N$ (fixed), such that $r_k$ is a strictly increasing sequence in $k$,\newline
(b) dist$(q_k, q_{k + 1}) > 2$ for all $k$,\newline
(c) $f(q_k) > 1$ for all $k$,\newline
(d) there exist annuli $D_k = (r_k - s_k, r_k + s'_k) \times N, s_k, s'_k > 0$ such that $f$ falls below $1/2$ somewhere inside each $D_k$ and the $D_k$'s do not intersect each other, and \newline
(e) $|D_k|$ is bounded above by a positive constant. \newline
Clearly\label{sym3}, 
\begin{align*}
\int_{D_k}|\nabla f|^2A(r)drdN & \geq C_k (\int_{D_k}|\nabla f|A(r)drdN)^2 \text{   (using (e))  }\\
& \gtrsim C_k(\int_{r_k - s_k}^{r_k + s'_k} |\nabla f|A(r)dr)^2 \gtrsim (\int_{r_k - s_k}^{r_k + s'_k} |\nabla f|dr)^2\\
& \gtrsim 1/4 \text{    (using (c) and (d))   }.
\end{align*}
where $C_k = \frac{1}{|D_k|}$ is bounded below, since $|D_k|$'s are bounded above. Since this is happening for all $k$, this will contradict the fact that $f \in H^1_r(M)$.\newline
We must point out that (d)  and (e) above hold necessarily, as otherwise, we will have a sequence of annuli $B_k$ such that $|f| > 1/2$ on $B_k$ and $|B_k| \to \infty$. That will imply $f \notin H^1_r(M)$.\newline 
\end{proof}
Here, we have assumed a lower bound on the function $A(r)$. To give some alternative criteria under which we can force $f$ to vanish at infinity, we refer to Lemma 2.1.1 from [MT], which says the following:
\begin{lemma} Assume that $A(r)$ satisfies either 
\[
\int_{|r| \geq 1}\frac{dr}{A(r)} < \infty
\]
or
\[
\lim_{|r| \to \infty} A(r) = \infty, \mbox{    and    } \sup_{|r| \geq 1}\bigg|\frac{A'(r)}{A(r)}\bigg| < \infty.
\]
Then
\[
f \in H^1_r(M) \Rightarrow f|_{M_1} \in C(M_1) \mbox{    and    }
\]
\[
\lim_{|r| \to \infty}|f(r)| = 0,
\]
where $M_1$ consists of all the points of $M$ having $r$-coordinates $\geq 1$.
\end{lemma}
Let us also prove the following
\begin{lemma}\label{noncpt}
Consider a non-compact manifold $M$ as described in the statement of Lemma \ref{Cond}. Given $m, \lambda \in \RR$, we assume 
the following bounds on $b$:
\beq\label{newb}
b^2 + 2|\lambda| b < 1,\text{  and also  } 2|\lambda|b < m^2 - \lambda^2 \text{    if   } m^2 - \lambda^2 > 0.
\eeq
Now, under (\ref{spec2}), and if (\ref{spec3}) holds, then we have,
\[
F_{m, \lambda, X}(u) \cong \V u\V^2_{H^1}.
\]
\end{lemma}
\begin{proof}
We have, $\text{Spec}(-\Delta + X^2 + 2i\lambda X) \subset [\beta(\lambda), \infty)$ and $m^2 - \lambda^2 > -\beta(\lambda)$. Assume first that $m^2 - \lambda^2 > 0$. We have,
\[
F_{m, \lambda, X}(u) \leq (-\Delta u, u) + |(Xu, Xu)| + 2|\lambda|| (Xu, u)| + ((m^2 - \lambda^2)u, u).
\] 
Using $\langle X, X\rangle \leq b^2$, on calculation this gives, $F_{m, \lambda, X}(u) \lesssim \V u\V^2_{H^1}$. Also,
\[
(-\Delta u, u) - |(Xu, Xu)| - 2|\lambda||(Xu, u)| + ((m^2 - \lambda^2)u, u)  \leq F_{m, \lambda, X}(u). \]
We want to show that 
\beq\label{HEAD}
C\V u\V^2_{H^1} \leq (-\Delta u, u) - |(Xu, Xu)| - 2|\lambda||(Xu, u)| + ((m^2 - \lambda^2)u, u),
\eeq
where $C > 0$ is independent of $u$. This will hold if and only if we can find a constant $C$ such that
\[
|(Xu, Xu)| + 2|\lambda||(Xu, u)| \leq (1 - C)(-\Delta u, u) + (m^2 - \lambda^2 - C)\V u\V^2_{L^2}.
\]
On calculation, using $\langle X, X\rangle \leq b^2$, we get
\[
|(Xu, Xu)| + 2|\lambda||(Xu, u)| \leq (b^2 + 2|\lambda|b)(-\Delta u, u) + 2|\lambda| b\V u\V^2_{L^2}.
\]
which finally proves (\ref{HEAD}).\newline
Now let us consider the case $0 \geq m^2 - \lambda^2 > -\beta(\lambda)$. As before, letting $-L_{2\lambda} = -\Delta + X^2 + 2i\lambda X$, we have
\[
F_{m, \lambda, X}(u) \leq (-L_{2\lambda} u, u) + (u, u) \lesssim \V u\V^2_{H^1} \text{  (using } \langle X, X\rangle \leq b^2).
\]
Also, the calculation for $\V u\V^2_{H^1} \lesssim (-L_{2\lambda} u, u) + (u, u)$ is similar to the proof of (\ref{HEAD}). So, we are done if we can prove that 
\beq\label{ACHE}
(-L_{2\lambda} u, u) + (u, u) \lesssim F_{m, \lambda, X}(u).
\eeq
We see that $(-L_{2\lambda} u, u) \geq \beta(\lambda)(u, u)$. When $\alpha > -\beta(\lambda)$, we have
\begin{align*}
(-L_{2\lambda} u, u) + \alpha (u, u) & \geq C(-L_{2\lambda} u, u) \geq \frac{C}{2}(-L_{2\lambda} u, u) + \frac{\beta(\lambda)}{2}C(u, u)\\
& \gtrsim (-L_{2\lambda} u, u ) + (u, u),
 \end{align*}
 where $C = 1 + \frac{\alpha}{\beta(\lambda)}$.
\end{proof}
Now, we have our first main theorem of this paper:
\begin{theorem}{\bf (Main theorem I)}\label{100}
Consider a non-compact manifold $M$ as described in the statement of Lemma \ref{Cond}. Given $m, \lambda \in \RR$, we assume (\ref{newb}) and (\ref{spec2}). Now, if (\ref{spec3}) is satisfied, then we can minimize $F_{m, \lambda, X}(u)$ in the class of functions $H^1_r(M)$ subject to (\ref{ic}). 
Here we keep $p$ in the range $(1, \frac{n + 2}{n - 2})$. 
\end{theorem}
\begin{proof}
We already know, under our assumptions, 
\begin{eqnarray}\label{a'}
F_{m, \lambda, X}(u) \cong \V u\V^2_{H^1(M)}.
\end{eqnarray}
We also have, $H^1(M) \hookrightarrow L^q(U)$ compactly, $q \in [2, \frac{2n}{n -2})$, where $\overline{U}$ is compact in $M$. Also, by Lemma \ref{Cond}, 
\[
u \in H^1_r(M) \Rightarrow u \mbox{ vanishes at infinity}.
\]
So, 
\[
u \in H^1_r(M) \Rightarrow u \in  L^\infty (M \setminus U).
\]
Also, $u \in L^2(M)$. This means, by interpolation, 
\[
u \in L^q(M \setminus U) \mbox{ for all } q \in [2, \infty].
\]
We also have,
\begin{align}\label{b'}
\int_{M\setminus U} |u|^q dM & \leq \V u\V^{q - 2}_{L^\infty (M \setminus U)}\int_{M\setminus U} |u|^2dM\nonumber \\
& \leq \V u\V^{q - 2}_{L^\infty (M \setminus U)}\V u\V^2_{H^1(M)},
\end{align}
and this gives,
\begin{align*}
u \in H^1_r(M) & \Rightarrow u \in L^q(M) \mbox{ } \forall \mbox{ } q \in [2, \frac{2n}{n - 2})\\
& \Rightarrow u \in L^{p + 1} (M) \mbox{ } \forall \mbox{ } p \in (1, \frac{n + 2}{n - 2}) \text{    } (p = 1 \text{  is not in our range}).
\end{align*}
As usual, let 
\[
I_\beta = \mbox{inf}\{F_{m, \lambda, X}(u) : u \in H^1_r(M), \text{      subject to   } (\ref{ic})\}.
\]
Clearly, $I_\beta > 0$, because of (\ref{a'}), (\ref{b'}) and the constraint (\ref{ic}).
Now, take a sequence $u_\nu \in H^1_r(M)$ such that $\V u_\nu\V^{p + 1}_{L^{p + 1}} = \beta$, and $F_{m, \lambda, X} (u_\nu) \leq I_\beta + 1/\nu$.\newline
Passing to a subsequence if necessary and without changing the notation, $u_\nu \rightarrow u \in H^1_r(M)$ weakly, which implies, by compact Sobolev embedding,
\begin{eqnarray}\label{shut}
u_\nu \longrightarrow u \mbox{ in } L^{p + 1}(U)\text{-norm} \mbox{ for all relatively compact } U.
\end{eqnarray}
Also, using (\ref{shut}) with very large $U$'s and the fact that $u_\nu, u$ vanish at infinity, we have from (\ref{b'}),
\begin{eqnarray}\label{up1}
\V u_\nu - u\V_{L^{p + 1}(M\setminus U)} \longrightarrow 0, \end{eqnarray}
meaning finally that 
\[
\V u\V^{p + 1}_{L^{p + 1}} = \beta.
\]
Also, we have to prove that $F_{m, \lambda, X}(u) = I_\beta$. This comes from the fact that
\[
F_{m, \lambda, X}(u) \leq \liminf F_{m, \lambda, X}(u_\nu).
\]
So finally a constrained $F_{m, \lambda, X}$ minimizer is obtained. 
\end{proof}
\subsection{Constrained energy minimizers}\label{MTIIS}
We now write about constrained energy minimizers in a non-compact setting. To be precise, we assume that our non-compact manifold $M$ is weakly homogeneous. To recall what this means, we make the following
\begin{definition}\label{WHS}
A manifold $M$ is said to be weakly homogeneous if there is a group $G$ of isometries of $M$ and a number $D > 0$ such that for every $x, y \in M$, there exists a $g \in G$ such that dist$(x, g(y)) \leq D$.
\end{definition} 
On such spaces, we are trying to minimize the energy 
\[
\mathcal{E}_{\lambda, X}(u) = \frac{1}{2}(-\Delta u + X^2u + 2i\lambda Xu, u) - \frac{1}{p + 1}\int_M|u|^{p + 1}dM
\]
subject to $\V u\V^2_{L^2} = \beta$ (constant) and (\ref{spec2}), the minimization being done over $H^1(M)$, and $p \in (1, 1 + 4/n)$.\newline
Let
\[
I_\beta = \mbox{inf}\{\mathcal{E}_{\lambda, X}(u) : u \in H^1(M), \V u\V^2_{L^2} = \beta\}.
\]
We will make the following technical assumption:
\begin{eqnarray}\label{ass}
I_\beta < -\frac{(m^2 - \lambda^2)}{2}\beta,
\end{eqnarray}
where $m$ is selected such that $m^2 - \lambda^2 > \text{max }\{-\beta(\lambda), 0\}$,  with $\beta(\lambda)$ defined as in (\ref{spec2}). We also assume that (\ref{newb}) is satisfied.\newline
With that in place, we state the second main theorem of this paper:
\begin{theorem}{\bf (Main Theorem II)}\label{1000}
If $M$ is a non-compact weakly homogeneous manifold, 
under the technical assumption (\ref{ass}), we can minimize $\Cal{E}_{\lambda, X}(u)$ in the class of functions $H^1(M)$ subject to $\Vert u\Vert_{L^2} = \beta$ and (\ref{spec2}). Here we want $p$ in the range $(1, 1 + \frac{4}{n})$.
\end{theorem}
Arguing with the Gagliardo-Nirenberg inequality as in Proposition 1.3, we can reach equation (\ref{broken1}), which lets us conclude that $I_\beta > -\infty$, and if $u_\nu$ is a sequence in $H^1(M)$ satisfying $\mathcal{E}_{\lambda, X}(u_\nu) < I_\beta + \frac{1}{\nu}$, then (a subsequence) $u_\nu$ is weakly convergent to $u \in H^1(M)$. \newline
Now, we can see that establishing $u$ as the constrained energy minimizer amounts to establishing two things:\newline
\begin{enumerate}
\item[$\bullet$] $u_\nu \longrightarrow u$ in $L^2$-norm, so that $\V u\V^2_{L^2} = \beta,$
\item[$\bullet$] $\mathcal{E}_{\lambda, X}(u) = I_\beta$.
\end{enumerate}
Now, in view of (\ref{EF}), the second bullet point will be established 
if we can prove that 
\beq\label{REQP}
\V u_\nu\V_{L^{p + 1}} \to \V u\V_{L^{p + 1}},
\eeq
and 
\beq \label{REQPI}
F_{m, \lambda, X}(u) \leq \liminf F_{m, \lambda, X}(u_\nu).
\eeq
Since $\V u_\nu\V_{H^1}$ is uniformly bounded, (\ref{REQP}) will be established via the Gagliardo-Nirenberg inequality (applied to $u - u_\nu$), in conjunction with the first bullet point above. Also, (\ref{REQPI}) is a consequence of weak convergence, spectral assumption (\ref{spec2}) and $m^2 - \lambda^2 > -\beta(\lambda)$.\newline
So now, our entire task hinges on proving the first bullet point, namely
\beq\label{VVI}
u_\nu \to u \text{   in  } L^2\text{-norm}.
\eeq
To accomplish this, we use the techniques of concentration-compactness, as laid out in ~\cite{L}. Below we give a formal statement of this. The statement was originally made in the setting of the Euclidean space, but as noted in ~\cite{CMMT}, the concentration-compactness principle and most of the subsidiary results generalize to manifolds of bounded geometry with essentially no changes at all. We will state the reformulated version as appears in ~\cite{CMMT}.
\begin{proposition}
Let $M$ be a Riemannian manifold with $C^\infty$ bounded geometry. Fix $\beta \in (0, \infty)$. Let $\{u_\nu\} \in L^{p + 1}(M)$ be a sequence satisfying $\int_M|u_\nu|^{p + 1}dM = \beta$. Then, after extracting a subsequence, one of the following three cases holds: \newline
(i) Vanishing: If $B_R(y) = \{x \in M : d(x, y) \leq R\}$ is the closed $R$-ball around $y$, then for all $R \in (0, \infty)$, 
\[
\lim_{\nu \to \infty}\sup_{y \in M}\int_{B_R(y)}|u_\nu|^{p + 1}dM = 0.
\]
(ii) Concentration: There exists a sequence of points $\{y_\nu\} \subset M$ with the property that for each $\varepsilon > 0$, there exists $R(\varepsilon) < \infty$ such that 
\[
\int_{B_{R(\varepsilon)}(y_\nu)}|u_\nu|^{p + 1}dM > \beta - \varepsilon.
\]
(iii) Splitting: There exists $\alpha \in (0, \beta)$ with the following properties: For each $\varepsilon > 0$, there exists $\nu_0 \geq 1$ and sets $E^{\#}_\nu, E^b_\nu \subset M$ such that 
\begin{eqnarray}\label{shutp}
d(E^{\#}_\nu, E^b_\nu) \rightarrow \infty \mbox{   as   } \nu \rightarrow \infty
\end{eqnarray}
and 
\begin{eqnarray}\label{shutq}
\bigg|\int_{E^{\#}_\nu}|u_\nu|^{p + 1}dM - \alpha\bigg| < \varepsilon, \bigg|\int_{E^{b}_\nu}|u_\nu|^{p + 1}dM - (\beta - \alpha)\bigg| < \varepsilon, \nu > \nu_0.
\end{eqnarray}

\end{proposition}
For a statement of the above fact in the even more general setting of measure metric spaces, see Appendix A.1 of ~\cite{CMMT}. A couple of lines about the heuristics of the concentration-compactness principle: when we have a sequence of elements in a Banach space with fixed norm, or, in other words, lying on a sphere in the Banach space, we cannot necessarily pick a norm convergent subsequence unless the Banach space itself is finite dimensional. But, we can give an exhaustive list of the possible behaviors of subsequences, at least in the context of the $L^p$ spaces. That is what the concentration-compactness principle gives. In our case, the only handle we have on the sequence $u_\nu$ is that all of them have the same $L^2$-norm. This should make the application of the concentration-compactness argument seem natural. In applications such as ours, the usual line of attack is to rule out vanishing and splitting phenomena, so we are left with concentration phenomenon as the only possibility. From there, we will show how to go to compactness, i.e., convergence of the subsequence, $\V u_\nu - u\V_{L^2} \rightarrow 0$, which has been the goal of the first bullet point. \newline
{\bf Ruling out vanishing and splitting}\newline
To rule out vanishing, one has to make the technical assumption mentioned before:
\begin{eqnarray}\label{ass1}
I_\beta < -\frac{(m^2 - \lambda^2)}{2}\beta,
\end{eqnarray}
where $m$ is selected such that $m^2 - \lambda^2 > -\beta(\lambda)$,  with $\beta(\lambda)$ defined as in (\ref{spec2}) and also $m^2 - \lambda^2 > 0$.\newline
It is not clear that we can always have (\ref{ass}) regardless of the manifold type. Some discussion about the assumption $I_\beta < 0$ is found in (3.0.10) and (3.0.11) of ~\cite{CMMT}. 
\newline 
Step I: Ruling out vanishing.\newline
Assume vanishing occurs, that is, $\forall\text{   } R \in (0, \infty)$,
\[
\lim_{\nu \to \infty}\sup_{y \in M}\int_{B_R(y)}|u_\nu|^{2}dM = 0 .
\]
We already know that $u_\nu$'s satisfy $\mathcal{E}_{\lambda, X}(u_\nu) < I_\beta + 1/\nu$ and that, $\{u_\nu\}$ is bounded in $H^1(M)$. \newline
Then, we have, by Lemma 2.1.2 of ~\cite{CMMT}
\[
2 < r < \frac{2n}{n - 2} \Longrightarrow \V u_\nu\V_{L^r(M)} \rightarrow 0.
\] 

That means, 
\[
\V u\V^2_{H^1} \cong F_{m, \lambda, X}(u) = 2\mathcal{E}_{\lambda, X}(u) + \frac{2}{p + 1}\int_M|u|^{p + 1}dM + (m^2 - \lambda^2)\beta
\]
implies in conjunction with (\ref{ass}) that
\[
\V u\V^2_{H^1} \leq \liminf \V u_\nu\V^2_{H^1} \leq \frac{2}{C^*}I_\beta + \frac{1}{C^*}(m^2 - \lambda^2)\beta < 0,
\]
which gives a contradiction. Here $C^*$ is a constant such that $C^*\V f\V^2_{H^1} \leq F_{m, \lambda, X}(f)$ for all $f \in H^1(M)$.\newline
Step II: Ruling out splitting.\newline
To rule out the splitting phenomenon, we first need a technical lemma, which is a special case of Propositions 3.1.2 and 3.1.3 of ~\cite{CMMT}.
\begin{lemma}\label{l1}
(i) If $\beta > 0, I_\beta < -\frac{m^2 - \lambda^2}{2}\beta, \sigma > 1$, then 
\begin{eqnarray}
I_{\sigma\beta} < \sigma I_\beta.
\end{eqnarray}
(ii) If $0 < \eta < \beta$ and $I_\beta < - \frac{m^2 - \lambda^2}{2}\beta$, we have 
\beq
I_\beta < I_{\beta - \eta} + I_{\eta}.
\eeq
\end{lemma}
Finally, we work to rule out splitting phenomena. We have
\begin{proposition}
If $\{u_\nu\} \in H^1(M)$ is a $\Cal{E}_{\lambda, X}$-minimizing sequence with $\V u_\nu\V^2_{L^2} = $ constant, then splitting ((\ref{shutp}) and (\ref{shutq})) cannot occur. 
\end{proposition}
\begin{proof} 
Begin by choosing $\varepsilon > 0$ sufficiently small such that 
\begin{eqnarray}\label{q}
I_{\beta} < I_\alpha + I_{\beta - \alpha} - C_1\varepsilon,
\end{eqnarray}
where $C_1$ is a constant that will be chosen later.\newline
Suppose now that splitting happens. We have already argued that $\V u_\nu\V_{H^1}$ is uniformly bounded. Also, seeing that $\V u_\nu\V_{L^2} = $ constant and $\V u_\nu\V_{L^{p + 1}}$ is uniformly bounded by an application of the Gagliardo-Nirenberg inequality, it follows from (\ref{shutp}) and  (\ref{shutq}) that there exists $\nu_1$ such that when $\nu \geq \nu_1$, we have
\beq\label{MASTER}
\int_{S_\nu}|u_\nu|^2dM + \int_{S_\nu}|\nabla u_\nu|^2dM + \int_{S_\nu}|u_\nu|^{p + 1}dM < \varepsilon,
\eeq
where $S_\nu$ is a set of the form 
\[
S_\nu = \{x \in M : d_\nu < d(x, E^{\#}_\nu) \leq d_\nu + 2\} \subset M\setminus (E^{\#}_\nu \cup E^b_\nu)
\]
for some $d_\nu > 0$. Call 
\[
\tilde{E}_\nu (r) = \{x \in M : d(x, E^{\#}_\nu) \leq r\}.
\]
Now define functions $\chi^{\#}_\nu$ and $\chi^b_\nu$ by
$$
\chi^{\#}_\nu (x) = 
\begin{cases}
1, & \text{if } x \in \tilde{E}_\nu(d_\nu)\\
1 - d(x, \tilde{\textsl{}E}_\nu(d_\nu)), & \text{if } x \in \tilde{E}_\nu(d_\nu + 1)\\
0, & \text{if } x \notin \tilde{E}_\nu(d_\nu + 1)
\end{cases}
$$
and 
$$
\chi^b_\nu (x) = 
\begin{cases}
0, & \text{if } x \in \tilde{E}_\nu(d_\nu + 1)\\
d(x, \tilde{E}_\nu(d_\nu + 1)), & \text{if } x \in \tilde{E}_\nu(d_\nu + 2)\\
1, & \text{if } x \notin \tilde{E}_\nu(d_\nu + 2).
\end{cases}
$$
Observe that both $\chi^{\#}_\nu(x)$ and $\chi^b_\nu(x)$ are Lipschitz with Lipschitz constant 1 and the intersection of their supports has measure zero. Also set
\[
u^{\#}_\nu = \chi^{\#}_\nu u_\nu, u^b_\nu = \chi^b_\nu u_\nu.
\]
Just to motivate what we are doing, we want a control on the term $|\mathcal{E}_{\lambda, X}(u_\nu) - \mathcal{E}_{\lambda, X}(u^{\#}_\nu + u^b_\nu)|$ i.e., show that 
\beq\label{TOSHOW}
|\mathcal{E}_{\lambda, X}(u_\nu) - \mathcal{E}_{\lambda, X}(u^{\#}_\nu + u^b_\nu)| = |\mathcal{E}_{\lambda, X}(u_\nu) - [\mathcal{E}_{\lambda, X}(u^{\#}_\nu) + \mathcal{E}_{\lambda, X}(u^b_\nu)]| \lesssim \varepsilon,
\eeq
and get a contradiction from the fact that $|I_\beta - I_\alpha - I_{\beta - \alpha}| > C_1\varepsilon$ which comes from (\ref{q}).
Choosing $m$ such that $m^2 - \lambda^2 > -\beta(\lambda)$, with $\beta(\lambda)$ as in (\ref{spec2}), we know that 
\[
2\mathcal{E}_{\lambda, X}(u_\nu) = F_{m, \lambda, X}(u_\nu) - \frac{2}{p + 1}\V u_\nu\V^{p + 1}_{L^{p + 1}} - (m^2 - \lambda^2)\V u_\nu\V^2_{L^2},
\]
and hence we see by triangle inequality that controlling each of the terms 
\beq
\int_M\bigg(|u_\nu|^{p + 1} - (|u^{\#}_\nu|^{p + 1} + |u_\nu|^{p + 1})\bigg)dM,
\eeq
\begin{eqnarray}\label{dhush}
\int_M\bigg(\V u_\nu\V_{H^1} - (\V u^{\#}_\nu\V_{H^1} + \V u^b_\nu\V_{H^1})\bigg)dM,
\end{eqnarray}
\begin{eqnarray}\label{jhamela}
|F_{m, \lambda, X}(u_\nu) - (F_{m, \lambda, X}(u^{\#}_\nu) + F_{m, \lambda, X}(u^b_\nu))|,
\end{eqnarray}
would be sufficient. To that end, we first 
note that when $\nu \geq \nu_1$, 
\[
\V u^{\#}_\nu\V^2_{L^2} = \alpha_\nu, \mbox{    where    } |\alpha - \alpha_\nu| < 2\varepsilon
\]
and
\begin{eqnarray*}
\V u^b_\nu\V^2_{L^2} = \beta_\nu - \alpha_\nu, \mbox{    where    } |(\beta - \alpha) - (\beta_\nu - \alpha_\nu)| < 2\varepsilon.
\end{eqnarray*}
Now, we have
\begin{eqnarray*}
\int_M\bigg(|u_\nu|^{p + 1} - (|u^{\#}_\nu|^{p + 1} + |u_\nu|^{p + 1})\bigg)dM \leq \int_{S_\nu}|u_\nu|^{p + 1}dM < \varepsilon
\end{eqnarray*}
and 
\begin{eqnarray}\label{1another}
\int_M\bigg(|u_\nu|^2 - (|u^{\#}_\nu|^2 + |u_\nu|^2)\bigg)dM \leq \int_{S_\nu}|u_\nu|^2dM < \varepsilon.
\end{eqnarray}
Using $\nabla u^{\#}_\nu = \chi^{\#}_\nu\nabla u_\nu + (\nabla \chi^{\#}_\nu)u_\nu$, the corresponding identity for $\nabla u^b_\nu$ and the fact that both $\chi^{\#}_\nu(x)$ and $\chi^b_\nu(x)$ have Lipschitz constant 1, we see that
\begin{align}\label{2}
\int_M\bigg(|\nabla u_\nu|^2 - (|\nabla u^{\#}_\nu|^2 + |\nabla u_\nu|^2)\bigg)dM & \leq \int_{S_\nu}|u_\nu|^2dM + \int_{S_\nu}|\nabla u_\nu|^2dM \lesssim \varepsilon.
\end{align}
(\ref{1another}) and (\ref{2}) together give (\ref{dhush}).
Now we are left with (\ref{jhamela}).\newline
From the definition of $F_{m, \lambda, X}(u)$ and what has gone before, we see that it suffices to control 
\begin{eqnarray*}
|(X^2 u_\nu, u_\nu) - (X^2 u^{\#}_\nu, u^{\#}_\nu) - (X^2 u^b_\nu, u^b_\nu)|
\end{eqnarray*}
or, equivalently,
\begin{eqnarray*}
\bigg|\int_M (|Xu_\nu|^2 - |Xu^{\#}_\nu|^2 - |Xu^b_\nu|^2)dM\bigg|.
\end{eqnarray*}
and also
\begin{eqnarray*}
\bigg|(iXu_\nu, u_\nu) - (iX u^{\#}_\nu, u^{\#}_\nu) - (iX u^b_\nu, u^b_\nu)\bigg|.
\end{eqnarray*} 
Now, as before, $Xu^{\#}_\nu = \chi^{\#}_\nu Xu^{\#}_\nu + X(\chi^{\#}_\nu)u^{\#}_\nu$, so
\begin{align}\label{above}
\bigg|\int_M (|Xu_\nu|^2 - |Xu^{\#}_\nu|^2 - |Xu^b_\nu|^2)dM\bigg| & \leq \int_{S_\nu}|u_\nu|^2dM + \int_{S_\nu}|X u_\nu|^2dM\\
& \lesssim \int_{S_\nu}|u_\nu|^2dM + \int_{S_\nu}|\nabla u_\nu|^2dM \lesssim \varepsilon,
\end{align}
the last observation coming from the fact that $X$ is bounded, which means that $|Xu_\nu| = |X.\nabla u_\nu| \lesssim |\nabla u_\nu|$. Lastly, we can also control 
\begin{eqnarray*}
\bigg|(iXu_\nu, u_\nu) - (iX u^{\#}_\nu, u^{\#}_\nu) - (iX u^b_\nu, u^b_\nu)\bigg|
\end{eqnarray*}
by using the Cauchy-Schwarz inequality. This is because
\begin{align*}
|(iXu_\nu, u_\nu) - (iXu^{\#}_\nu, u^{\#}_\nu) - (iXu^{b}_\nu, u^{b}_\nu)|  = |(Xu_\nu, u_\nu) - (\chi^{\#}_\nu Xu^{\#}_\nu + X(\chi^{\#}_\nu)u^{\#}_\nu, u^{\#}_\nu)\\
- (\chi^{b}_\nu Xu^b_\nu + X(\chi^{b}_\nu)u_\nu^b, u^b_\nu)|\\
= |(Xu_\nu, u_\nu) - (Xu_\nu, \chi^{\#}_\nu u_\nu) - (Xu_\nu, \chi^b_\nu u_\nu) - (X(\chi_\nu^{\#})u_\nu, u_\nu^{\#}) + (X(\chi_\nu^{b})u_\nu, u_\nu^{b})|
\end{align*}
\begin{align*}
\leq |(Xu_\nu, u_\nu) - (Xu_\nu, \chi^{\#}_\nu u_\nu) - (Xu_\nu, \chi^b_\nu u_\nu)| + |(X(\chi_\nu^{\#})u_\nu, u_\nu^{\#})| + |(X(\chi_\nu^{b})u_\nu, u_\nu^{b})|\\
\leq |\int_{S_\nu} Xu_\nu\overline{u_\nu}| + |\int_{S_\nu} u_\nu\overline{u_\nu^{\#}}| + |\int_{S_\nu}u_\nu\overline{u_\nu^b}| \leq |\int_{S_\nu} Xu_\nu\overline{u_\nu}| + 2\int_{S_\nu}|u_\nu|^2.
\end{align*} 
Now,
\[
|\int_{S_\nu} Xu_\nu \overline{u_\nu}| \leq \V Xu_\nu\V_{L^2({S_\nu})}\V u_\nu\V_{L^2({S_\nu})} \leq \V X.\nabla u_\nu\V_{L^2({S_\nu})}\V u_\nu\V_{L^2({S_\nu})} \]\[\lesssim \V\nabla u_\nu\V_{L^2({S_\nu})}\V u_\nu\V_{L^2({S_\nu})} \leq \varepsilon
\]
the last step coming from (\ref{MASTER}).
That completes the proof.
\end{proof} 
Now that we have ruled out the alternatives, we can say that the minimizing sequence $u_\nu$ will concentrate. Recall that this means
\begin{corollary}\label{cor}
Under the setting of Lemma \ref{l1}, there is a sequence of points $y_\nu \in M$  such that for all $\varepsilon > 0$, there exists $R(\varepsilon) < \infty$ (independent of $\nu$) such that 
\begin{eqnarray}\label{quote}
\int_{M \setminus B_{R(\varepsilon)}(y_\nu)} |u_\nu|^2 dM < \varepsilon.
\end{eqnarray}
\end{corollary}
This allows us to invoke the assumption of weak homogeneity at last. Using weak homogeneity, we can map the sequence $y_\nu$ into a compact region $K \subset M$ and we still call the translates of $u_\nu$ as $u_\nu$. Now, any subsequence which concentrates will have compact Sobolev embedding, i.e., we use the compact embedding $H^1(M) \hookrightarrow L^2(K)$. Then, weak $H^1$ convergence of $u_\nu$ allows us to find a subsequence, still called $u_\nu$, such that 
\beq\label{LMom}
\V u_\nu - u\V_{L^2(K)} \to 0.
\eeq
We see that (\ref{LMom}) holds for all bounded $K \subset M$. Hence, from (\ref{quote}), we see that $\V u\V^2_{L^2} = \beta$ and $u_\nu \to u$ in $L^2(M)$-norm.\newline
This is what we intended to prove.\newline

\section{\bf Small perturbations of $X$ and corresponding minimisers}\label{P}
Previously, we have established the existence of $F_{m, \lambda, X}$ minimizers on non-compact manifolds of type $M = N \times [0, \infty)$ which have bounded geometry and metrics of the type $g = dr^2 + \phi (r)g_N$, with smooth $\phi$. Similarly, we can establish the existence of constrained $F_{\lambda, X}$ minimizers.\newline
Now, we raise the following perturbation question: if we perturb the Killing field $X$ slightly, can we prove that the corresponding constrained minimizers also vary slightly? It certainly seems believable on a compact manifold, but the question is more involved on a non-compact setting. We study this question for the $F_{\lambda, X}$ minimizers in connection with the NLS equation. Roughly, we establish that on the class of manifolds considered in Lemma \ref{Cond}, slight perturbations in the Killing field will result in slight variations in the constrained $F_{\lambda, X}$-minimizers in the sense of $L^s$-norm, $s \in [2, \frac{2n}{n - 2})$. More formally, 
\begin{proposition}\label{SP}
Consider a non-compact manifold $M$ satisfying all the conditions as in Lemma \ref{Cond}. Start with a bounded Killing field $X$ on $M$ and assume that $X_n = X + \varepsilon_n X^{''}$ is a Killing field for all $n \in \mathbb{N}$, where $X^{''}$ is another bounded Killing field on $M$, and $\varepsilon_n \to 0$ is a decreasing sequence of positive real numbers. Let $u_n$ be constrained minimizers of $F_{\lambda, X_n}(u)$ subject to (\ref{ic}). Then, there exists a subsequence of $u_n$, still called $u_n$, such that 
\[
\V u_m - u_n\V_{L^s} \to 0, \text{   as  } m, n \to \infty, s \in [2, \frac{2n}{n - 2}).
\]
\end{proposition}
\begin{proof}
Let us suppose, $X^{''}$ is such that $\langle X^{''}, X^{''}\rangle \leq C$. \newline
Now
\begin{align*}
F_{\lambda, X^{'}}(u) & = (-\Delta u + \lambda u - iX^{'} u, u) \\
& = (-\Delta u - iXu - i\varepsilon X^{''} u + \lambda u, u)\\
& = F_{\lambda, X}(u) - (i\varepsilon X^{''}u, u) \longrightarrow F_{\lambda, X}(u) \mbox{  as  } \varepsilon \longrightarrow 0.
\end{align*}
Also,
\begin{align*}
|F_{\lambda, X^{'}}(u) - F_{\lambda, X}(u)| & = |(i\varepsilon X^{''}u, u)| \leq \V \varepsilon X^{''} u\V\V u\V\\
& = \V\varepsilon X^{''}.\nabla u\V\V u\V \leq |\varepsilon X^{''}|\V \nabla u\V \V u\V \\
& \leq C\varepsilon \V u\V^2_{H^1},
\end{align*}
the last step using the bound on $X^{''}$ and the Cauchy-Schwarz inequality.\newline
Now, consider a decreasing sequence $\varepsilon_n \longrightarrow 0$ and consider the perturbations $X_n = X + \varepsilon_n X^{''}$ of $X$. Suppose for each $n$, $u_n \in H^1_r(M)$ is a minimizer of $F_{\lambda, X_n}$ subject to $\V u\V^{p + 1}_{L^{p + 1}} = \beta$ (constant). We will start by arguing that $u_n$ has a convergent subsequence in the $L^r$-norm, $r \in [2, \frac{2n}{n - 2})$. \newline
First, we argue that $\V u_n\V^2_{H^1}$ is uniformly bounded. We need to see that 
\begin{align} \label{b}
I_{\beta, n} \mbox{   uniformly bounded   } & \Longleftrightarrow F_{\lambda, X_n}(u_n) \mbox{   uniformly bounded   }\\ & \Longrightarrow \V u_n\V^2_{H^1} \mbox{   uniformly bounded},
 \end{align}
where $I_{\beta, n} = \mbox{inf}\{F_{\lambda, X_n} (u) : u \in H^1_r(M), \V u\V^{p + 1}_{L^{p + 1}} = \beta\}$.\newline
We observe that $F_{\lambda, X_n}(u) \geq 0$ for all $n, u$. So, fixing an integer $k$, we have
\begin{align}
|I_{\beta, k + q} - I_{\beta, k}| & \leq |F_{\lambda, X_{k + q}}(u_k) - F_{\lambda, X_{k}}(u_k)|\\
& \lesssim \varepsilon_k \V u_k\V^2_{H^1}
\end{align}
for all positive integers $q$. That means, 
\begin{align}
I_{\beta, k + q} & \leq I_{\beta, k} + C\varepsilon_k \V u_k\V^2_{H^1},
\end{align}
which gives that $I_{\beta, n}$ is uniformly bounded, which means that $F_{\lambda, X_n}(u_n)$ is uniformly bounded. \newline
Now we know, $\V u\V^2_{H^1(M)} \leq C^{(n)}F_{\lambda, X_n}(u)$ for all $u \in H^1(M)$. That gives, for $m > k$, $k$ being fixed,
\begin{align*}
\V u\V^2_{H^1(M)} & \leq C^{(k)}((-\Delta - iX_k + \lambda)u, u)\\
& = C^{(k)}((-\Delta - iX_m + \lambda)u, u) + C^{(k)}((iX_m - iX_k)u, u)\\
& \leq C^{(k)}F_{\lambda, X_m}(u) + C^{(k)}|(\varepsilon_m - \varepsilon_k)(X^{''}u, u)|\\
& \leq C^{(k)}F_{\lambda, X_m}(u) + CC^{(k)}\varepsilon_k\V u\V^2_{H^1(M)},
\end{align*}
which finally implies
\beq\label{lageraho}
\V u\V^2_{H^1} \leq \frac{C^{(k)}F_{\lambda, X_m}(u)}{1 - CC^{(k)}\varepsilon_k}.
\eeq
Let us justify (\ref{lageraho}). We claim that there exist some positive integer $l$ such that $CC^{(k)}\varepsilon_k < 1$ for all $k \geq l$. If not, there will exist a subsequence, still called $C^{(k)}$ with mild abuse of notation, such that $CC^{(k)}\varepsilon_k \geq 1$, which means that $C^{(k)} \to \infty$ as $k \to \infty$. But this contradicts the fact that $C^{(k)} \to C^*$, where $C^*$ is a constant such that 
\[
\V u\V^2_{H^1} \leq C^*F_{\lambda, X}(u).\]
Now, (\ref{lageraho}) means, in particular
\beq
\V u_m\V^2_{H^1} \leq \frac{C^{(k)}}{1 - CC^{(k)}\varepsilon_k}F_{\lambda, X_m}(u_m),
\eeq
which means that finally we have, $\{\V u_n\V^2_{H^1}\}$ is uniformly bounded.\newline
Since we have $\V u_n\V^2_{H^1} \leq K$ uniformly, we can say that $u_n$ converges in the weak topology of $H^1(M)$. Since we are working on manifolds of bounded geometry, 
we also have, when $2 < s < \infty$, and relatively compact $U$,
\begin{align}\label{c}
\int_{M\setminus U} |u_n|^s dM & \leq \V u_n\V^{s - 2}_{L^\infty (M \setminus U)}\int_{M\setminus U} |u_n|^2dM \\
& \leq \V u_n\V^{s - 2}_{L^\infty (M \setminus U)}\V u_n\V^2_{H^1(M)}.
\end{align}
Also, compact Sobolev embedding gives us
\[
H^1(M) \hookrightarrow L^s(\Omega) \text{   compactly,  } \forall s \in [2, \frac{2n}{n - 2})
\]
given $\Omega \subset M$ relatively compact. This, together with (\ref{c}), gives\begin{align*}
u_n \in H^1_r(M) & \Rightarrow u_n \in L^s(M) \mbox{ } \forall \mbox{ } s \in [2, \frac{2n}{n - 2}).
\end{align*}
Passing to a subsequence if necessary and without changing the notation, $u_n \rightarrow u \in H^1_r(M)$ weakly implies, by the compactness of Sobolev embedding,
\begin{eqnarray}\label{st}
u_n \longrightarrow u \mbox{ in } L^s(U)-\mbox{norm} \mbox{ for all relatively compact } U.
\end{eqnarray}
Also, using (\ref{c}), using (\ref{st}) with very large $U$'s and the fact that $u_n, u$ vanish at infinity (this is Lemma \ref{Cond}) , we have 
\begin{eqnarray}\label{up}
u_n \longrightarrow u \mbox{   in   } L^s(M\setminus U)-\mbox{norm},
\end{eqnarray}
meaning, 
$\V u_{k + q} - u_k\V_{L^s}$ is small beyond some integer $k$, for all positive integers $q$.\newline

\end{proof}
\subsection{Acknowledgements}
I would like to thank my Ph.D. advisor Michael Taylor for introducing me to this project and also guiding me throughout with academic advice. I am also indebted to Patrick Eberlein for several illuminating discussions on geometry and to Jeremy Marzuola for kindly going through a draft copy of this paper. 
\section{Appendix}
Here we justify our claim that on $S^n$, $-L_{2\lambda} = -\Delta + X^2 + 2i\lambda X$ loses one derivative when $|\lambda| < \frac{n - 1}{2}$, i.e., $L_{2\lambda}u \in H^s_{\text{loc}} \implies u \in H^{s + 1}_{\text{loc}}$. To prove this, we invoke Theorem 1.8, Chapter XV of ~\cite{T2}. Let $P = p(x, D) \in OPS^2$ have an expansion into homogeneous terms as follows:
\[
p(x, \xi) \sim p_2(x, \xi) + p_1(x, \xi) + ..... \]
Define the subprincipal symbol of $P$ as follows:
\[
\text{sub }\sigma(P) = p_1(x, \xi) + \frac{i}{2}\sum_\nu \frac{\pa^2}{\pa x_\nu\pa \xi_\nu}p_2(x, \xi).\]
Also, let $\Sigma = \{(x, \xi) : p_2(x, \xi) = 0\}$. Let $Q = Q_{(x, \xi)}$ denote the Hessian of $p_2(x, \xi)$ at a point $(x, \xi) \in \Sigma$. $Q = Q(u)$ is a quadratic form on $T(T^*(X))$, and we denote the associated bilinear form by $Q(u, v)$. If $\sigma$ denotes the symplectic form on $T^*(X)$, we define the associated Hamilton map $F$ by
\[
\sigma (u, Fv) = Q(u, v).\]
It can be proved that the nonzero eigenvalues of $F$ are of the form $\pm i\mu_\nu (\mu_\nu > 0),$ and call $tr^+ F := \sum \mu_\nu$.\newline
Now we can quote the aforementioned theorem:
\begin{theorem}(~\cite{T2}) Suppose $p_2(x, \xi) \geq 0$ vanishes to exactly second order on $\Sigma$, with $T\Sigma$ not involutive at each point of $\Sigma$, and suppose 
\beq\label{ebar}
\text{Re sub }\sigma (P) + tr^{+}F > 0
\eeq
holds. Then $P$ is hypoelliptic with loss of exactly one derivative.
\end{theorem} 
From the above statement, observe that it is immediately clear that $L_0$ loses exactly one derivative, because (\ref{ebar}) is trivially satisfied in that case.\newline 
The principal symbol of $L_{2\lambda}$ is given by 
\[
p_2(x, \xi) = \sum_{i = 2}^{n}\xi_i^2 + \sum_{i < j}(x_i\xi_j - x_j\xi_i)^2\]
The characteristic variety is 
\[\Sigma = \{((x_1, x_2,...,x_n); (0, 0,..., 0))\}\cup\{((x_1, 0,..., 0); (\xi_1, 0,..., 0))\}.
\]
So, 
\[
\text{dist }(((x_1, 0,.., 0); (\xi_1, \xi_2,..., \xi_n)), \Sigma)^2 = \sum_{j = 2}^n \xi_j^2 \leq p_2(x, \xi),
\]
which shows that $p_2(x, \xi)$ vanishes to exactly second order on $\Sigma$.\newline
For the rest of the calculation, we specialize to $n = 2$ for ease of presentation, because the general case is exactly analogous. We already have $\text{Re sub }\sigma (P) = -2\lambda\xi_1$. We can see that
\[
Q = \left( \begin{array}{cccc}
0 & 0 & 0 & 0 \\
0 & 2\xi_1^2 & 0 & -2x_1\xi_1\\
0 & 0 & 0 & 0\\
0 & -2x_1\xi_1 & 0 & 2x_1^2 + 2
\end{array}\right),\]
which gives on calculation that 
\[
F = \left( \begin{array}{cccc}
0 & 0 & 0 & 0 \\
0 & -2x_1\xi_1 & 0 & 2x_1^2 + 2\\
0 & 0 & 0 & 0\\
0 & -2\xi_1^2 & 0 & 2x_1\xi_1
\end{array}\right),\]
Calculating the eigenvalues of $F$, we see that $tr^{+}F = 2\vert\xi_1\vert$. So, if $|\lambda| < 1/2$, we have that 
\[
\text{Re sub }\sigma (P) + tr^{+}F > 0.
\]
\bibliographystyle{plain}
\def\noopsort#1{}

\end{document}